\documentclass[11pt,a4]{article}
\usepackage[margin=1in]{geometry}
\usepackage{amsthm, amsfonts,amsmath,amssymb,bbm}
\usepackage{verbatim}
\usepackage{textgreek}
\usepackage[]{authblk}
\usepackage{color}
\usepackage{hyperref}
\usepackage{etoolbox}
\makeatletter
\let\ams@starttoc\@starttoc
\makeatother
\usepackage[parfill]{parskip}
\makeatletter
\let\@starttoc\ams@starttoc
\patchcmd{\@starttoc}{\makeatletter}{\makeatletter\parskip\z@}{}{}
\makeatother

\usepackage{enumerate}
\usepackage{pdfsync}
\usepackage{color}

\numberwithin{equation}{section}

  \newtheorem{theorem}{Theorem}[section]
  
  \newtheorem{lemma}[theorem]{Lemma}
   \newtheorem{corollary}[theorem]{Corollary}

\theoremstyle{definition}

\theoremstyle{remark}
\newtheorem{remark}[theorem]{Remark}

\newcommand{\lb}{\left (}
\newcommand{\rb}{\right )}
\newcommand{\lbb}{\left [}
\newcommand{\rbb}{\right ]}
\newcommand{\labs}{\left |}
\newcommand{\rabs}{\right |}
\newcommand{\lbrb}[1]{\lb #1 \rb}
\newcommand{\lbbrbb}[1]{\lbb#1\rbb}

\newcommand{\lbbrb}[1]{\lbb#1\rb}

\newcommand{\lbcurly}{\left\{}
\newcommand{\rbcurly}{\right\}}

\newcommand{\intervalOI}{\lbrb{0,\infty}}
\newcommand{\abs}[1]{\labs#1\rabs}

\newcommand{\curly}[1]{\lbcurly#1\rbcurly}

\newcommand{\so}[1]{\mathrm{o}\lbrb{#1}}

\newcommand{\Pbb}[1]{\Pb\lb #1\rb}
\newcommand{\Ebb}[1]{\Eb\lbb #1\rbb}

\newcommand{\LL}{L\'{e}vy }
\newcommand{\LLP}{L\'{e}vy process }
\newcommand{\LLPs}{L\'{e}vy processes }
\newcommand{\LLK}{\LL\!\!-Khintchine }

\newcommand{\limi}[1]{\lim\limits_{#1\to \infty}}
\newcommand{\limsupi}[1]{\varlimsup\limits_{#1\to \infty}}
\newcommand{\liminfi}[1]{\varliminf\limits_{#1\to \infty}}
\newcommand{\limo}[1]{\lim\limits_{#1\to 0}}

\newcommand{\Cb}{\mathbb{C}}

\newcommand{\Eb}{\mathbb{E}}

\newcommand{\Rb}{\mathbb{R}}

\newcommand{\Rp}{\mathbb{R}^+}

\newcommand{\Pb}{\mathbb{P}}

\newcommand{\Zb}{\mathbb{Z}}

\newcommand{\Mcc}{\mathcal{M}}

\newcommand{\ind}[1]{\mathbb{I}_{\{#1\}}}

\newcommand{\IntOI}{\int_{0}^{\infty}}
\newcommand{\IntII}{\int_{-\infty}^{\infty}}

\newcommand{\Wp}{W_\phi}
\newcommand{\Wkap}{W_\kappa}
\newcommand{\WkapP}{W_{\kappa_+}}
\newcommand{\WkapN}{W_{\kappa_-}}

\newcommand{\IPsi}{I_\Psi}
\newcommand{\IPsih}{I_{\hat{\Psi}}}

\newcommand{\MPsi}{\Mcc_{I_\Psi}}
\newcommand{\MPsih}{\Mcc_{\IPsih}}

\newcommand{\ab}{a+ib}

\renewcommand{\Im}{\mathtt{Im}}
\renewcommand{\Re}{\mathtt{Re}}
\newcommand{\filtration}[2]{\lbrb{#1_{#2}}_{#2\geq 0}}

\title{Moments of exponential functionals of L\'{e}vy processes on a deterministic horizon -- identities and explicit expressions}
\author{Z. Palmowski\footnote{Department of Applied Mathematics, Faculty of Pure and Applied Mathematics,
Wroclaw University of Science and Technology
ul. Janiszewskiego 14a, 50-372 Wroclaw, Poland, $^{\textnormal{a}}$\href{mailto:zbigniew.palmowski@pwr.edu.pl}{zbigniew.palmowski@pwr.edu.pl}}\hspace{0.12cm}$^{,\textnormal{a}}$, H. Sariev\footnote{Institute of Mathematics and Informatics, Bulgarian Academy of Sciences, Georgi Bonchev str. bl.8, Sofia 1113, Bulgaria, $^{\textnormal{b}}$\href{mailto:h.sariev@math.bas.bg}{h.sariev@math.bas.bg}}\hspace{0.12cm}$^{,\ddagger,\textnormal{b}}$ and M. Savov\footnote{Faculty of Mathematics and Informatics, Sofia University "St. Kliment Ohridski", 5 James Bourchier blvd., Sofia 1164, Bulgaria, $^{\textnormal{c}}$\href{mailto:msavov@fmi.uni-sofia.bg}{msavov@fmi.uni-sofia.bg}}\hspace{0.12cm}$^{,\dagger,\textnormal{c}}$}

\date{}

\begin{document}

\maketitle
\begin{abstract}
In this work, we consider moments of exponential functionals of L\'{e}vy processes on a deterministic horizon. We derive two convolutional identities regarding these moments. The first one relates the complex moments of the exponential functional of a general L\'{e}vy process up to a deterministic time to those of the dual L\'{e}vy process. The second convolutional identity links the complex moments of the exponential functional of a L\'{e}vy process, which is not a compound Poisson process, to those of the exponential functionals of its ascending/descending ladder heights on a random horizon determined by the respective local times. As a consequence, we derive a universal expression for the half-negative moment of the exponential functional of any symmetric L\'{e}vy process, which resembles in its universality the passage time of symmetric random walks. The $(n-1/2)^{th}$, $n\geq 0$ moments are also discussed. On the other hand, under extremely mild conditions, we obtain a series expansion for the complex moments (including those with negative real part) of the exponential functionals of subordinators. This significantly extends previous results and offers neat expressions for the negative real moments. In a special case, it turns out that the Riemann zeta function is the minus first moment of the exponential functional of the Gamma subordinator indexed in time.
\end{abstract}


\section{Introduction, background and motivation}

Let $\xi=\filtration{\xi}{t}$ be a potentially killed \LLP that takes values on the real line and possesses \LLK exponent $\Psi$. Denote by
\[I_{\Psi}(t):=\int_0^t e^{-\xi_s}ds, \quad t\in(0,\infty],\]
the associated exponential functional on a deterministic horizon. When $t=\mathbf{e}_q$ is an exponentially distributed random variable with parameter $q>0$, independent of $\xi$, the functional-analytic properties of $I_\Psi(\mathbf{e}_q)$, including its moments, have been extensively studied; a small list of references containing general results and techniques is \cite{BerYor05,Maulik-Zwart-06,PardoPatieSavov2012,PatieSavov2012,PatieSavov2018,Urban95,Yor01}. The number of papers containing applications of $I_\Psi(\mathbf{e}_q)$ is enormous, some of which are discussed in \cite{MinSav23,PatieSavov2018}. The functional-analytical properties of $\IPsi(t)$, for $t\in(0,\infty)$, are much less understood. To the best of our knowledge, the only case where a relatively explicit expression for the density of $\IPsi(t)$ is available is the case where $\xi$ is Brownian motion with drift, see \cite{AlMatSh01,Yor92} and the survey by \cite{MatYor05} for more details. When $\xi$ is a subordinator, the moments of $\IPsi(t)$ are somewhat better understood, with some noteworthy references being \cite{BarkerSavov2021,SalVos18, SalVos19, vostrikova_2020}. In fact, \cite{SalVos18,SalVos19} deal with exponential functionals of additive processes and derive recurrent relations between their moments, providing some explicit formulae for the positive integer moments of $\IPsi(t)$ when $\xi$ is either a subordinator, Brownian motion with drift, or a compound Poisson process with standard normal jumps, see \cite[Corollary 1, Examples 5-6]{SalVos18}. Stemming from the formula in \cite{FitPit99} for the integer moments of Markov additive functionals, a multiple integral expression for $\Eb[\IPsi^n(t)]$, $n\geq 0$ is given in \cite{Br22,SalVos19}, with the motivation coming from applications to Asian options, see also \cite{HackKuz14,Pat13}. Curiously, as opposed to $I_\Psi(\mathbf{e}_q)$, which appears in various theoretical studies, the interest in $I_\Psi(t)$ seems to originate predominantly from applications.

The first aim of this paper is to establish two convolutional identities involving the moments of $\IPsi(t)$. The first one, see Lemma \ref{lem:convo}, relates the moments  of $\IPsi(t)$ to those of $\IPsih(t)$, where $\hat{\Psi}$ is the \LLK exponent of 
the dual \LL process $\hat{\xi}=(-\xi_t)_{t\geq 0}$. For symmetric \LLPs with infinite activity, this relation yields the universal expression $\Eb[\IPsi^{-1/2}(t)]=t^{-1/2}$, with the same one holding for diffuse \LLPs of finite activity. So far this identity was known only for Brownian motion, see \cite{Yor92}. In Section \ref{subsec:examples}, we discuss the $(n-1/2)^{th}$ moments of symmetric \LL processes. The second convolutional identity, see Lemma \ref{lem:convo2}, relates  the moments of $\IPsi(t)$ to the moments of the ladder height processes of $\xi$ evaluated at the respective local times. We have pointed to a lengthy and technical route by which the moments of the latter can be expanded into a series, which is a promising way to study the moments of $\IPsi(t)$ itself.

The second aim of this paper is to significantly improve \cite{BarkerSavov2021} and provide a series expansion for all moments, including negative ones, of $\IPsi(t)$ when $\xi$ is a subordinator. This is done under conditions that are much weaker than the already mild assumptions in \cite[Theorem 2.14]{BarkerSavov2021}. Besides general complex moments, our Theorem \ref{thm:NegMom} offers relatively neat expressions for the negative integer moments of $\IPsi(t)$. In a special case, $\Eb[\IPsi^{-1}(t)]=\zeta(t+1)$, where $\zeta$ is the celebrated Riemann zeta function, see Subsection \ref{subsec:examples}. Our result also gives the speed of convergence of the series, thus allowing for suitable truncation and approximation.

The paper is structured as follows: Section \ref{sec:main} introduces notation, presents the main results of our study, offers examples and discusses the applicability of our results;  Section \ref{sec:proofs} contains the proofs.

\section{Main results}\label{sec:main}

\subsection{Preliminary definitions and notation}

We use $\Rb$, $\Rp=\lbbrb{0,\infty}$, $\Cb$ to denote the real line, the half-line, and the complex plane, respectively. For any $z\in\Cb$, we write $z=\ab=\Re(z)+i\Im(z)$.

Recall that a one-dimensional \LL process, $\xi=\lbrb{\xi_t}_{t\geq 0}$, is a real-valued stochastic process defined in a suitable probability space that has stationary and independent increments. In this work, we allow for an independent exponential killing of $\xi$ in the usual fashion: if $\mathbf{e}_q$ is an exponentially distributed random variable with parameter $q\geq0$, independent of $\xi$, we set $\xi_t=\infty$ for $t\geq\mathbf{e}_q$. When we employ a L\'{e}vy process, we implicitly understand that it may be killed in the described way. This extension is useful in the context of exponential functionals on a deterministic horizon, since the killing is in fact their Laplace transform. We recall that the \LLK exponent $\Psi$ of any $\xi$ is given by
\begin{equation}\label{def:Psi}
\Psi(z):=\log\Ebb{e^{z\xi_1}}=az+\frac{\sigma^2}{2}z^2+\IntII\lbrb{e^{zy}-1-zy\ind{|y|\leq 1}}\Pi(dy)-q,\quad z\in i\Rb,
\end{equation}
where $a\in\Rb$, $\sigma^2\geq 0$, $q\geq 0$ are the linear term, the Brownian component, and the killing rate, respectively, $\Pi$ is a sigma-finite measure satisfying $\IntII \min\curly{x^2,1}\Pi(dy)<\infty$ and encoding the intensity and size of the jumps of $\xi$, and $\ind{\cdot}$ stands for the indicator function. Note that when $q=0$, then $\xi$ is conservative, that is, it is not killed. For any $\Psi$, it is well-known that the following Wiener-Hopf factorization holds
\begin{equation}\label{def:WH}
\Psi(z)=h(q)\kappa_+(q,-z)\kappa_-(q,z),\quad z\in i\Rb,
\end{equation}
where $\kappa_\pm$ are the bivariate Bernstein functions of the ascending/descending bivariate subordinators
related to the inverse ascending/descending ladder times and the ascending/descending ladder heights. More precisely, if $L_\pm=\lbrb{L_\pm(t)}_{t\geq 0}$ are the local times of the strong Markov processes $(\xi^+_t)_{t\geq 0}=(\sup_{s\leq t}\xi_s-\xi_t)_{t\geq 0}$ and $(\xi^-_t)_{t\geq 0}=(\xi_t-\inf_{s\leq t}\xi_s)_{t\geq 0}$, then $L^{-1}_\pm=(L^{-1}_\pm(t))_{t\geq 0}$ are the ascending/descending ladder times, $H_\pm=\lbrb{H_\pm(t)}_{t\geq 0}=(\xi_{L^{-1}_\pm(t)})_{t\geq 0}$ are the ascending/descending ladder heights, and
\[-\log\Ebb{e^{-qL^{-1}_\pm(1)-zH_\pm(t)}}=\kappa_\pm(q,z)=e^{\IntOI \int_{\lbbrb{0,\infty}} \lbrb{e^{-t}-e^{-qt-zx}}\frac{\Pbb{\xi_t\in \pm dx}}{t}dt},\quad \Re(z)\geq 0,\]
are the bivariate Bernstein functions or, in the language of L\'{e}vy processes, the bivariate Laplace exponents of $(L^{-1}_\pm,H_\pm)$, refer to \cite[Chapter VI]{Ber96} for more information. In addition,
\begin{equation}\label{eq:h}
	\begin{split}
		&h(q)=e^{-\int_{0}^{\infty}\lbrb{e^{-t}-e^{-qt}}\Pbb{\xi_t=0}\frac{dt}{t}},
	\end{split}
\end{equation}
with $h(q)\equiv1$ whenever $\xi$ is not a compound Poisson process, i.e. a \LLP whose intensity of jumps is finite $(\Pi\lbrb{\Rb}<\infty)$.
Furthermore, we can associate to each bivariate Bernstein function $\kappa$ a bivariate Bernstein-Gamma function $\Wkap$, which for any fixed $\zeta\in\Cb$ with $\Re(\zeta)\geq 0$ is the solution of
\begin{equation}\label{def:W}
\Wkap\lbrb{\zeta,z+1}=\kappa(\zeta,z)\Wkap(\zeta,z),\quad \Wkap(\zeta,1)=1,\quad \text{ on $\Re(z)>0$.}
\end{equation}
In fact, for any $q=\zeta\in\Rp$, $\Wkap(q,\cdot)$ is the Mellin transform of a positive random variable, that is, there exists $Y\geq0$ such that $\Eb[Y^z]=\Wkap(q,z+1)$, for $\Re(z)>-1$, see \cite[Theorem 2.8]{BarkerSavov2021}. The functional-analytic properties, including Stirling-type asymptotics, of the bivariate Bernstein-Gamma function $\Wkap$ are studied in detail in \cite{BarkerSavov2021}. Various similar results in the univariate case can be found in \cite{MinSav23,PatieSavov2013,PatieSavov2018,PatSav21}.

For any $\xi$ and $t\geq 0$, we consider the exponential functionals
\begin{equation*}
    \begin{split}
    &I_{\Psi}(t):=\int_0^t e^{-\xi_s}ds=\int_0^{\min\curly{t,\mathbf{e}_q}} e^{-\xi_s}ds,\\
    &I_{\Psi}\equiv I_{\Psi}(\infty):=\IntOI e^{-\xi_s}ds=\int_0^{\mathbf{e}_q}e^{-\xi_s}ds.
    \end{split}
\end{equation*}
By the Strong Law of Large Numbers, $\IPsi<\infty$ if and only if $q>0$ or $\lim_{s\rightarrow\infty}\xi_s=\infty$ almost surely. We introduce the Mellin transform of $I_{\Psi}$, defined formally for some $z\in \Cb$ by
\begin{equation*}
\begin{split}
\MPsi(z):=\Eb\bigl[I_\Psi^{z-1}\bigr].
\end{split}
\end{equation*}
When $h(q)\equiv1$, it follows from Theorem 2.4 in \cite{PatieSavov2018} that
\begin{equation}\label{eq:recur}
    \MPsi(z)=\kappa_-(q,0)\frac{\Gamma(z)}{\WkapP(q,z)}\WkapN(q,1-z),\quad  \Re(z)\in(0,1),
\end{equation}
since, in the notation of \cite[Lemma A1]{PatieSavov2018}, we have $\kappa_{\pm}=\phi_\pm$. Indeed, from \cite[Lemma A1]{PatieSavov2018}, for any fixed $q\geq 0$, it holds $\phi_+(z)=h(q)\kappa_+(q,z)$, and so from \cite[Theorem 4.1.(5)]{PatieSavov2018}, we obtain \eqref{eq:recur}. Relation \eqref{eq:recur} has been at the heart of the most refined results in the area of exponential functionals, which is mainly due to the fact that the expression holds at least on a complex strip of unit length and the asymptotics of $\Wkap$ are available in the complex half-plane. The most precise version of the latter is contained in \cite{MinSav23}.

\subsection{General convolutional identities}

In this part, we consider general convolutional identities regarding moments of exponential functionals on a deterministic horizon. The first one connects the moments of $\IPsi(t)$  and $\IPsih(t)$, where $\hat{\Psi}(z)=\Psi(-z)$ is the \LLK exponent of the dual \LLP $\hat{\xi}=(-\xi_t)_{t\geq 0}$. This allows us to find a universal formula for the $-1/2^{th}$-moment of the exponential functional of any symmetric \LL process.\\

\begin{lemma}\label{lem:convo}
Let $\Psi$ and $\hat{\Psi}$ be the \LLK exponents of the \LLPs $\xi$ and $\hat{\xi}$ that are not compound Poisson processes. Then, for any $t>0$ and $\Re(z)\in\lbrb{0,1}$, we have
 \begin{equation}\label{eq:convo}
     \int_{0}^t\Ebb{\IPsi^{-z}(t-s)}\Ebb{I^{z-1}_{\hat{\Psi}}(s)}ds=\Gamma(1-z)\Gamma(z)=\frac{\pi}{\sin{\pi z}}.
 \end{equation}
If $\xi$ is a compound Poisson process with continuous measure $\Pi(dx)$ and $\lambda=\Pi\lbrb{\Rb}$, then
\begin{equation}\label{eq:convo1}
     \int_{0}^t\Ebb{\IPsi^{-z}(t-s)}\Ebb{I^{z-1}_{\hat{\Psi}}(s)}ds=\Gamma(1-z)\Gamma(z)\frac{(\lambda +1)^2}{\lambda^2}\lbrb{1-e^{-\lambda t}-\lambda te^{-\lambda t}}\ind{t>0}.
\end{equation}
\end{lemma}

\begin{remark}\label{rem:convo}
It would not be surprising if \eqref{eq:convo} could be proved from general considerations involving properties of \LL processes such as stationarity and independence of increments. However, expression \eqref{eq:recur} for the Mellin transform clearly suggests the identity and the proof is straightforward.
\end{remark}

We offer an immediate corollary that reveals an interesting invariance in the case of symmetric L\'{e}vy processes, which is reminiscent of the law of the entrance time for symmetric random walks, see \eqref{eq:convoSS}.\\

\begin{corollary}\label{cor:convo}
Let $\xi$ be a symmetric L\'{e}vy process, so that $\xi\stackrel{w}{=}\hat{\xi}$ and $\Psi\equiv\hat{\Psi}$. If $\xi$ is not a compound Poisson process, then we have, for any $t>0$, 
\begin{equation}\label{eq:convoS}
     \int_{0}^t\Ebb{\IPsi^{-z}(t-s)}\Ebb{I^{z-1}_{\Psi}(s)}ds=\Gamma(1-z)\Gamma(z)=\frac{\pi }{\sin{\pi z}},
\end{equation}
and
\begin{equation}\label{eq:convoSS}
\Eb\Bigl[I^{-1/2}_{\Psi}(t)\Bigr]=t^{-1/2}.
\end{equation}
If $\xi$ is a compound Poisson process with continuous measure $\Pi(dx)$ and $\lambda=\Pi\lbrb{\Rb}$, then, for any $t>0$,
\begin{equation}\label{eq:convoS1}
     \int_{0}^t\Ebb{\IPsi^{-z}(t-s)}\Ebb{I^{z-1}_{\hat{\Psi}}(s)}ds=\Gamma(1-z)\Gamma(z)\frac{(\lambda +1)^2}{\lambda^2}\lbrb{1-e^{-\lambda t}-\lambda te^{-\lambda t}}\ind{t>0},
\end{equation}
and
\begin{equation}\label{eq:convoSS1}
\Eb\Bigl[I^{-1/2}_{\Psi}(t)\Bigr]=\lbrb{\lambda+1}\sqrt{\pi}e^{-\lambda t}\int_0^t \frac{e^{\lambda s}}{\sqrt{s}}ds.
\end{equation}
\end{corollary}

\begin{remark}
Identity \eqref{eq:convoSS} appears already in \cite{Yor92} in the case of Brownian motion.
\end{remark}

The second identity is again convolutional, but does not involve simultaneously $\IPsi(t)$ and $\IPsih(t)$. Before stating it, we relate the ratio in \eqref{eq:recur} to a suitable Laplace transform. This is of independent interest as the left-hand side of \eqref{eq:LTrep} appears in the representation of the Mellin transform of $\IPsi$, see \cite{PatieSavov2018}.\\

\begin{lemma}\label{lem:LTrep}
Let $\kappa$ be the Laplace exponent of a bivariate subordinator which is a Wiener-Hopf factor of the \LLK exponent of a \LLP $\xi$. Then, for any $q>0$, it holds that
\begin{equation}\label{eq:LTrep}
\frac{1}{q}\frac{\Gamma(z)}{W_{\kappa}(q,z)}=\IntOI e^{-qt} \Ebb{I^{z-1}_\kappa\lbrb{L(t)}}dt,\quad \Re(z)>0,
\end{equation}
where $I_{\kappa}\lbrb{t}=\int_{0}^t e^{-H(s)}ds$, $(L^{-1}(s), H(s))_{s\geq 0}$ is the bivariate ladder height process related to $\kappa$, and $L(t)$ is the local time related to the ladder time $L^{-1}(t)$.
\end{lemma}

This result paves the way for the second convolutional identity.\\

\begin{lemma}\label{lem:convo2}
Let $\Psi$ be the \LLK exponent of a \LL process $\xi$ that is not a compound Poisson process. Then, for any $t>0$ and $\Re(z)\in\lbrb{0,1}$, it holds that
\begin{equation}\label{eq:convo2}
    \int_0^t \Ebb{\IPsi^{z-1}(t-s)}\Ebb{I^{-z}_{\kappa_-}(L_-(s))}ds dt=\Gamma\lbrb{1-z}\int_{0}^t U_{L^{-1}_+}(dt-s) \Ebb{I^{z-1}_{\kappa_+}(L_+(s))}ds,
\end{equation}
where $I_{\kappa_\pm}(t)=\int_0^t e^{-H_\pm(s)}ds$, $t\geq 0$ are the exponential functionals of the ladder height subordinators $H_\pm$, $L_\pm$ are  the local times of the ladder times $L^{-1}_\pm$, and $U_{L^{-1}_\pm}$ are the respective potential measures.
\end{lemma}

Next, we provide a corollary in a special case.\\ 

\begin{corollary}\label{cor:convo2}
Let $\Psi$ be the \LLK exponent of a symmetric \LL process $\xi$ that is not a compound Poisson process. Then $\kappa_\pm=\kappa$, and $L_\pm=L$ and $L^{-1}_\pm=L^{-1}$ are stable subordinators of index $1/2$, and we have, for any $\Re(z)\in\lbrb{0,1}$,
\begin{equation}\label{eq:convo2_1}
    \int_0^t \Ebb{\IPsi^{z-1}(t-s)}\Ebb{I^{-z}_{\kappa}(L(s))}ds =\int_{0}^t \frac{\Gamma\lbrb{1-z}}{\Gamma(\frac12)\sqrt{t-s}} \Ebb{I^{z-1}_{\kappa}(L(s))}ds.
\end{equation}
\end{corollary}

\begin{remark}\label{rem:convo2}
At a convolutional level, the moments of $\IPsi(t)$ are related to the moments of $I_\kappa(\cdot)$ evaluated at the local time $L(v)$. Let $\xi$ be standard Brownian motion. Then $I_{\kappa}(L(v))=1-e^{-L(v)}$ and one may proceed to recover classical results for the moments of the exponential functional of Brownian motion, see \cite{Yor92}. In general, if we set $z=1/2$, then $\Eb[I^{-1/2}_\kappa(L(v))]$ appears in both convolutions in \eqref{eq:convo2_1} and we recover $\Eb[\IPsi^{-1/2}(t)]=t^{-1/2}$, see \eqref{eq:convoSS}. Unfortunately, this argument fails even for $z=1/2+ib$, since then we have an identity of the type $f_1*g=f_2*\overline{g}$.
\end{remark}

\begin{remark}\label{rem:convo3}
Formulae \eqref{eq:convo2} and \eqref{eq:convo2_1} can be useful for understanding $\Eb[\IPsi^{z}(t)]$, provided a series expansion in the spirit of \eqref{eq:NegMom} for $\Ebb{I^{z}_{\kappa}(L(v))}$ is derived. This would be a highly technical but plausible endeavor, as can be seen from the details in the proofs in \cite{BarkerSavov2021} and later when we deal with the easiest case, namely $(L^{-1}(t),H(t))=(t,H(t))$. We leave this for future work.
\end{remark}

\subsection{Moments of exponential functionals of subordinators}

The main result in this part extends and expands Theorem 2.14 in \cite{BarkerSavov2021} and yields an infinite series representation for $\Eb[I^{z}_{\phi}(t)]=\int_0^t e^{-\xi_s}ds$,
for a large class of subordinators $\xi$. Note that the Laplace exponents of non-decreasing \LLPs are in bijection with Bernstein functions via the relation
\begin{equation*}\label{definitionphisubordinator}
\phi(z)=-\log\Ebb{e^{-z\xi_1}}=-\Psi(-z),\quad \Re(z)\geq 0,
\end{equation*}
hence, the use of $I_\phi$ instead of $I_\Psi$. We recall that a function $\phi$ is a Bernstein function if and only if
\begin{equation}\label{def:Bern}
\phi(z)=\phi(0)+dz+\IntOI\lbrb{1-e^{-zy}}\mu(dy),\quad \Re(z)\geq 0,
\end{equation}
where $\phi(0)$, $d\in\Rp$, and $\mu$ is a sigma-finite measure on $\intervalOI$ satisfying $\IntOI \min\curly{y,1}\mu(dy)<\infty$. Following \cite[(2.5)-(2.7)]{PatieSavov2018}, for any Bernstein function $\phi$, we set
\begin{equation}\label{eq:aphi}
    \begin{split}
        \mathbf{a}_\phi&=\inf\curly{u<0:\phi\in \mathbf{A}_{\lbrb{u,\infty}}}\in[-\infty,0],\\
        \mathbf{u}_\phi&=\inf\curly{u\in\lbbrbb{\mathbf{a}_\phi,0}: \phi(u)=0}\in[-\infty,0],\\
        \mathbf{\bar{a}}_\phi&=\max\curly{ \mathbf{a}_\phi,\mathbf{u}_\phi}\in[-\infty,0],
    \end{split}
\end{equation}
where we use the convention $\sup\,\emptyset = -\infty$ and $\inf\,\emptyset= 0$, and $\mathbf{A}_{(u,\infty)}$, $u\in\Rb$ is the space of holomorphic functions on $\curly{z\in\Cb: \Re(z)>u}$.

The Bernstein-gamma function $W_\phi$, being the univariate analogue of $\Wkap$ (see \eqref{def:W}), is the unique solution in the space of Mellin transforms of positive random variables of the equation
\begin{equation}\label{Bernsteingammafunction}
W_\phi(z+1)=\phi(z)W_\phi(z), \quad W_\phi(1)=1, \quad \Re(z)\geq 0.
\end{equation}

For the next theorem, denote $\limsupi{x} =\limsup_{x\to\infty}$ and $\liminfi{x} =\liminf_{x\to\infty}$.\\

\begin{theorem}\label{thm:NegMom}
Let $\phi$ be any Bernstein function, i.e. the Laplace exponent of a potentially killed subordinator. Let $\limsupi{x}\phi'(x)/\phi'(2x)<\infty$, $\limi{x}x^2\phi'(x)=\infty$ and $\phi(\infty)=\infty$. Then, for all $t>0$ and $\Re(z)>-1+\mathbf{a}_\phi\ind{\phi(0)=0}$, we have
\begin{equation}\label{eq:NegMom}
\Ebb{I^{z}_{\phi}(t)}=\frac{\Gamma(z+1)}{W_{\phi}(z+1)}-\sum_{k\geq 1}\frac{\prod_{j=1}^k\lbrb{\phi(z+j)-\phi(k)}}{\prod_{j=1}^{k-1}\lbrb{\phi(j)-\phi(k)}}\frac{e^{-\phi(k)t}}{\phi(k)} \frac{\Gamma(z+1)}{W_{\phi_{(k)}}(z+1)},
\end{equation}
where $\phi_{(n)}(z):=\phi(n+z)-\phi(n)$, $n\geq1$ is a Bernstein function.
\end{theorem}~ 

\begin{remark}\label{rem:NegMom}
Let us compare our result to Theorem 2.14 in \cite{BarkerSavov2021}, which requires
\begin{equation}\label{suffcond}
\liminfi{x}\phi'(\lambda x)/\phi'(x)\geq C\lambda^{\beta},\qquad \beta>-1,
\end{equation}
uniformly on $\lambda\in[1,\Lambda]$, for $\Lambda>1$. Under \eqref{suffcond}, it is immediate that $\overline{\lim}_{x\rightarrow\infty}\phi'(x)/\phi'(2x)<\infty$. Moreover, $\phi'(x)\geq Cx^{\rho-1}$, for some $\rho>0$ and all $x$ large enough, see \cite[discussion around (6.16)]{BarkerSavov2021}. Thus, $\lim_{x\rightarrow\infty}x^2\phi'(x)=\infty$ and the conditions of Theorem \ref{thm:NegMom} are satisfied. To illustrate the scope of our result, take $\phi(x)=\ln\ln(x+e)$, which is a Bernstein function. This is an example from \cite[Remark 2.15]{BarkerSavov2021} and for which Theorem 2.14 in \cite{BarkerSavov2021} is not applicable. Then $\phi'(x)=1/\lbrb{\lbrb{x+e}\ln\lbrb{x+e}}$ and our conditions immediately hold. The same is true for more pathological cases such as $\phi(x)=\ln\ln\ln(x+e^e)$ and $\phi(x)=\ln\ln\ln\ln(x+e^{e{^e}})$, which are completely intractable by \cite[Theorem 2.14]{BarkerSavov2021}. The steps in the proof of Theorem \ref{thm:NegMom} coincide with those in \cite{BarkerSavov2021}, however, we had to significantly strengthen most of the bounds borrowed from there, e.g. via a new approximation of the key products in Lemma \ref{lem:H}.
\end{remark}~ 

\begin{remark}\label{rem:relation}
We present some conditions on the \LL measure $\mu$  and the drift $d$ in the definition \eqref{def:Bern} of $\phi$, which yield the conditions for $\phi$ in Theorem \ref{thm:NegMom}. First, note that $\phi(\infty)=\infty$ is equivalent to $\mu(\Rb^+)=\infty$ or $d>0$. If $d>0$, then $\phi'(\infty)=d>0$ and our conditions hold. Let $d=0$. Then, setting $\bar{\mu}(x):=\int_{x}^\infty \mu(dy)$, sufficient conditions in terms of $\mu$ for $\limi{x}x^2\phi'(x)=\infty$ come from
\[x^2\phi'(x)=x^2\IntOI ye^{-xy}\mu(dy)\geq x^2e^{-1}\int_{0}^{\frac1x}y\mu(dy)\geq \frac{e^{-1}}2 x\lbrb{\bar{\mu}\lbrb{\frac{1}{2x}}-\bar{\mu}\lbrb{\frac{1}{x}}},\]
with $\lim_{x\rightarrow\infty}x^2\int_{0}^{1/x}y\mu(dy)=\infty$ in all cases we tested. Let us consider $\overline{\lim}_{x\rightarrow\infty}\phi'(x)/\phi'(2x)<\infty$. From the fact that $\phi'$ is decreasing and $x\phi'(x)\leq \phi(x)$, see \cite[Proposition 3.1, (3.3)]{BarkerSavov2021}, we see that our condition is implied by $\underline{\lim}_{x\rightarrow\infty}\phi(2x)/\phi(x)>1$, which follows, for example, by $\underline{\lim}_{x\rightarrow 0+}\bar{\mu}(x/2)/\bar{\mu}(x)>1$. We highlight that the latter is too stringent and other sufficient conditions can be inferred from \cite{grzywny_leżaj_trojan_2023}.
\end{remark}~ 

\begin{remark}\label{rem:NegMom1}
Note that in this work we also extend the range over which \eqref{eq:NegMom} is valid. Our result holds for $\Re(z)>-1+\mathbf{a}_\phi\ind{\phi(0)=0}$ as opposed to Theorem 2.14 in \cite{BarkerSavov2021}, where only $\Re(z)>0$ is considered.
\end{remark}~ 

\begin{remark}\label{rem:NegMom2}
Exponential functionals of \LLPs are related to Ornstein-Uhlenbeck processes. For example, $\IPsi$ is the stationary law of a subclass of such processes, see \cite{BehLinMal_11,BehLinRek21,KuzParSav2012}, which is a consequence of the fact that $\IPsi(t)$ describes the evolution of this subclass of Ornstein-Uhlenbeck processes, see \cite{BehLinMal_11, Ber19}. Hence, positive moments of exponential functionals are related to those of Ornstein-Uhlenbeck processes, whereas negative ones are linked to a suitable Feller semigroup, see \cite{Patie08}.
\end{remark}~ 

\begin{remark}\label{rem:Zb}
In \cite{SalVos18, SalVos19, vostrikova_2020}, the identification of the moments is based on recurrent formulas that can be used for numerical purposes. Observe that \eqref{eq:NegMom} can also be used for this purpose as long as $W_{\phi}$ and $W_{\phi_{(k)}}$ are found. To that end, positive integer moments can be computed via the recursive formula \eqref{Bernsteingammafunction}, yielding $W_\phi(n)=\prod_{k=1}^n \phi(k)$, while any $W_\phi(z)$, for $\Re(z)> 0$, can be represented as an infinite Weierstrass product involving $\phi$ (see, e.g., \cite{PatieSavov2018}),
\[W_\phi(z)=\frac{e^{-\gamma_\phi z}}{\phi(z)}\prod_{k=1}^\infty \frac{\phi(k)}{\phi(k+z)}e^{\frac{\phi^\prime(k)}{\phi(k)}z},\]
where $\gamma_\phi=\lim_{n\to \infty}\bigl(\sum_{k=1}^n \frac{\phi^\prime(k)}{\phi(k)}-\log \phi(n)\bigr)$. Other semi-explicit expressions for $W_\phi$ can be found in \cite[Section 4]{HirschYor}, see also \cite[Theorem 4.7]{PatieSavov2018}. In addition, the q-gamma function and the double Gamma functions are instances of $\Wp$, see the introduction of \cite{PatieSavov2018}. Furthermore, for general $\Wp$ there are Stirling type approximations, see  \cite{BarkerSavov2021,MinSav23}, which can be used for numerical purposes too.
\end{remark}

When $\phi(0)=0$ and $\mathbf{a}_\phi<0$, we can evaluate at least some negative integer moments of $I_\phi(t)$.\\ 

\begin{corollary}\label{cor:NegMom}
Let $\phi$ be a Bernstein function such that $\phi(0)=0$ and $\mathbf{a}_\phi<0$. Suppose $\limsupi{x}\phi'(x)/\phi'(2x)<\infty$, $\limi{x}x^2\phi'(x)=\infty$ and $\phi(\infty)=\infty$. Then, for any $l\in\Zb$ such that $-1+\mathbf{a}_\phi<l<0$, we have, for $t>0$,
\begingroup\allowdisplaybreaks
\begin{align}
\Ebb{I^{l}_{\phi}(t)}&=\phi'(0)\frac{\prod_{j=1}^{-l-1}-\phi(j+l)}{(-l-1)!}+\sum_{k\geq 1}\frac{\prod_{j=1}^{-l-1}\lbrb{\phi(k)-\phi(j+l)}}{(-l-1)!}\phi'(k)e^{-\phi(k)t} \nonumber\\
&=\sum_{k\geq 0}\frac{\prod_{j=1}^{-l-1}\lbrb{\phi(k)-\phi(j+l)}}{(-l-1)!}\phi'(k)e^{-\phi(k)t},\label{eq:NegMom1}
\end{align}
\endgroup
where the existence of $\phi'(0)$ comes from the fact that $\mathbf{a}_\phi<0$.
\end{corollary}~ 

\begin{remark}\label{rem:NegMom3}
The relation \eqref{eq:NegMom1} holds always for $l=-1$, in which case we get the neat expression
\begin{equation}\label{eq:NegMom1_1}
\Ebb{I^{-1}_{\phi}(t)}=\sum_{k\geq 0}\phi'(k)e^{-\phi(k)t}=\Ebb{\xi_1}+\Eb\Biggl[\frac{\xi_te^{-\xi_t}}{1-e^{-\xi_t}}\Biggr],
\end{equation}
where the last part stems from $\Eb[\xi_te^{-k\xi_t}]=\phi'(k)e^{-\phi(k)}$. If $\mathbf{a}_\phi<-1$, then using \eqref{eq:NegMom1} and \eqref{eq:NegMom1_1}, the second negative moment is given by the expression
\begin{equation*}\label{eq:NegMom1_2}
\Ebb{I^{-2}_{\phi}(t)}=\sum_{k\geq 0}\phi'(k)\lbrb{\phi(k)-\phi(-1)}e^{-\phi(k)t}=-\frac{d}{dt}\Ebb{I^{-1}_{\phi}(t)}-\phi(-1)\Eb\Bigl[I^{-1}_{\phi}(t)\Bigr].
\end{equation*}
The last equation also appears in \cite[Theorem 2 (42)]{SalVos18} and we have found its solution here. In fact, Theorem 2 in \cite{SalVos18} provides general recurrent relations for the negative moments of $I_\phi(t)$, whose solutions are given by \eqref{eq:NegMom1}.
\end{remark}

\subsection{Illustrative examples and discussion}\label{subsec:examples}
We start with some more consequences of Lemma \ref{lem:convo} regarding the moments of $I_\Psi(t)$ of order $n-1/2$, $n\geq 1$.  Let us consider an unkilled symmetric \LL process $\xi$ with \LLK exponent $\Psi$ that is not a compound Poisson process. Assume that $\Psi(-1/2-\epsilon)<\infty$ for some $\epsilon>0$, and note that, from Theorem 2.4 in \cite{PatieSavov2018}, \eqref{eq:recur} and Corollary \ref{cor:convo}, we get for $q>\Psi(-1/2)=\Psi(1/2)$ that
\[\IntOI e^{-qt}\Ebb{\IPsi^{\frac12}(t)}dt=\frac{1/2}{q-\Psi\lbrb{-\frac12}}\IntOI e^{-qt}\Ebb{\IPsi^{-\frac12}(t)}dt=\frac{1/2}{q-\Psi\lbrb{\frac12}}\frac{\Gamma\lbrb{\frac{1}{2}}}{\sqrt{q}}.\]
On the right-hand side we have a product of Laplace transforms, and since the identity holds for all $q>\Psi(-1/2)=\Psi(1/2)$, we get
\[\Ebb{\IPsi^{\frac12}(t)}=\frac{e^{\Psi\lbrb{\frac12}t}}{2}\int_{0}^{t}e^{-\Psi\lbrb{\frac12}s}\frac{ds}{\sqrt{s}}.\]
Assuming that $\Psi\lbrb{n-1/2+\varepsilon}=\Psi\lbrb{-n+1/2-\varepsilon}<\infty$, for some $\varepsilon>0$, and applying recursively Theorem 2.4 in \cite{PatieSavov2018}, we obtain for $q>\Psi\lbrb{n-1/2}$
\[\IntOI e^{-qt}\Ebb{\IPsi^{n-\frac12}(t)}dt=\lbrb{\prod_{k=1}^n\frac{k-\frac{1}{2}}{q-\Psi\lbrb{k-\frac12}}}\frac{\Gamma\lbrb{\frac{1}{2}}}{\sqrt{q}}.\]
As a result, $\Eb [\IPsi^{n-1/2}(t)]$ is an $(n+1)$-fold convolution of the measures $t^{-1/2}dt$ and $\nu_k(dt)=(k-1/2)e^{\Psi\lbrb{k-1/2}t}dt$, $1\leq k\leq n$, considered on $\Rb^+$.

We proceed with some examples stemming from Theorem \ref{thm:NegMom}. To do so, we introduce a couple of Bernstein functions that will be used in the sequel. Set
\begin{equation*}
    \begin{split}       \phi^{\mathbf{ln}}(x)=\ln\lbrb{x+1},\quad\phi_\alpha(x)=x^\alpha,\; \alpha\in(0,1),\quad\phi^{\mathbf{ln_2}}(x)=\ln\ln(x+e).
    \end{split}
\end{equation*}
Let us first discuss $W_{\phi_{(k)}},k\geq 1$, which appears in Theorem \ref{thm:NegMom}. Note that from \eqref{def:Bern},
\[\phi_{(k)}(z):=\phi(k+z)-\phi(k)=dz+\IntOI (1-e^{-zy})e^{-ky}\mu(dy)\]
are Bernstein functions with \LL measure $e^{-ky}\mu(dy)$ and the same drift. In general, there is no explicit way to relate $W_{\phi_{(k)}}$ to $\Wp$ apart from using the general Stirling approximation, mentioned in Remark \ref{rem:Zb}, and linking it to the approximation of $\Wp$. Nevertheless, consider $\phi_\alpha$ and set $\phi^*(x)=\phi_{(1),\alpha}(x)=(x+1)^\alpha-1$. Put $\phi_{(k),\alpha}(z)=k^\alpha\phi^*(z/k)=k^\alpha \phi^*_{(k)}(z)$, with $\phi^*_{(k)}$ being a Bernstein function. From \cite[Theorem 4.1 (5)]{PatieSavov2018}, we have that $W_{\phi_{(k),\alpha}}(z)=k^{\alpha(z-1)}W_{\phi^*_{(k)}}(z)$, so from \eqref{Bernsteingammafunction} we conclude that $W_{\phi^*_{(k)}}(z)=W_{\phi^*}(z/k)$. Therefore, the series \eqref{eq:NegMom} depends only on $W_{\phi_\alpha}$ and $W_{\phi^*}$. On the other hand, for $\phi^{\mathbf{ln}}$, we see that $\phi^{\mathbf{ln}}_{(k)}(x)=\ln\lbrb{x/(k+1)+1}=\phi^{\mathbf{ln}}(x/(k+1))$, in which case the series \eqref{eq:NegMom} depends only on $\Wp$.

We now give particular examples of moments computed using Theorem \ref{thm:NegMom}. Note that $\phi^{\mathbf{ln}}$ is the exponent of the Gamma subordinator with \LL measure $\mu(dy)=y^{-1}e^{-y}dy$, and by Remark \ref{rem:NegMom}, it satisfies the assumptions of Theorem \ref{thm:NegMom}. Since $\phi^{\mathbf{ln}}(0)=0$, Corollary \ref{cor:NegMom} also holds with $\mathbf{a}_{\phi^{\mathbf{ln}}}=-1$. It follows that
\begin{equation*}
\begin{split}
\Ebb{I^{-1}_{\phi^{\mathbf{ln}}}(t)}=\sum_{k\geq 1}\frac{1}{k^{t+1}}=\zeta(t+1),
\end{split}
\end{equation*}
which is the celebrated Riemann zeta function. Observe from \eqref{eq:NegMom1_1} that the Riemann zeta function can be represented as the expectation of a particular function of the Gamma subordinator at a fixed point,
\[\zeta(t+1)=\Ebb{\xi_1}+\Ebb{\frac{\xi_te^{-\xi_t}}{1-e^{-\xi_t}}},\]
from which, using that $\xi$ is a subordinator, we can easily check that $\zeta$ has a simple pole at $1$, since $\lim_{t\rightarrow0}\Eb\bigl[\frac{\xi_te^{-\xi_t}}{1-e^{-\xi_t}}\bigr]=1$. In fact, this also follows from \cite[Theorem 2.18, (2.40)]{PatieSavov2018}, which asserts that $\lim_{t\rightarrow0}\Eb[I^{-1}_{\phi^{\mathbf{ln}}}(t)]=1$. In that case, the $-2^{\text{nd}}$ moment is not well-defined.

Consider the subordinator with the truncated L\'evy measure $\mu_1(dy)=y^{-1}e^{-y}\ind{y<1}dy$ and  Laplace exponent $\phi_1(x)$. Then \[\phi_1'(x)=\lbrb{1-e^{-x-1}}/(x+1) \text{ and } \phi_1(x)=\int_0^x\frac{1-e^{-y-1}}{y+1}dy=\phi^{\mathbf{ln}}(x)-A(x),\]
where $A(x)=\int_0^x\frac{e^{-y-1}}{y+1}dy$; hence, $\Eb[I^{-1}_{\phi_1}(t)]=\sum_{k\geq 0}\frac{1-e^{-k-1}}{(k+1)^{t+1}}e^{tA(k)}$. Since $\mathbf{a}_{\phi_1}=-\infty$, we can evaluate all negative moments, the second being
\[\Ebb{I^{-2}_{\phi_1}(t)}=\sum_{k\geq 0}\frac{1-e^{-k-1}}{(k+1)^{t+1}}\lbrb{\ln{\lbrb{k+1}}-A(k)+\int_{0}^{1}\frac{1-e^{-y}}{y}dy}e^{tA(k)}.\]
Finally, for $\phi^{\mathbf{ln_2}}$, we can easily calculate that $\Eb[I^{-1}_{\phi^{\mathbf{ln_2}}}(t)]=\sum_{k\geq 0}[(k+e)\ln^{t+1}(k+e)]^{-1}$. One may construct many similar computable examples where the approximations are good enough.

The series expansion also offers the speed of convergence, as $t\rightarrow\infty$. If $z=\ab$, then we can employ \eqref{eq:estimate2} to estimate the tail of the series. More precisely, for any $l>1$, we get from Lemma \ref{lem:H}
\begin{equation*}\label{example:NegMom}
\begin{split}
    &e^{\phi(l+1)t}\abs{\Ebb{I^{z}_{\phi}(t)}-\frac{\Gamma(z+1)}{W_{\phi}(z+1)}+\sum_{k\leq l}\frac{\prod_{j=1}^k\lbrb{\phi(z+j)-\phi(k)}}{\prod_{j=1}^{k-1}\lbrb{\phi(j)-\phi(k)}}\frac{e^{-\phi(k)t}}{\phi(k)} \frac{\Gamma(z+1)}{W_{\phi_{(k)}}(z+1)}}\\
    &\leq K\sum_{k\geq l+1}\frac{\phi'(k)}{\lbrb{\phi(k)}^{1+a}}e^{c\lbrb{|a|+\frac{a^2}{2}}\phi(k)\lbrb{\frac{8\ln{k}}{k^2\phi'(k)}+\frac{\varphi'(y_k)}{\varphi^2(y_k)}}}e^{-(\phi(k)-\phi(l+1))t},
\end{split}
\end{equation*}
where the series on the right-hand side is convergent under the conditions of Theorem \ref{thm:NegMom}.

\section{Proofs}\label{sec:proofs}

We start with the proof Lemma \ref{lem:convo}.

\begin{proof}[Proof of Lemma \ref{lem:convo}]
We use \eqref{eq:recur} to get that, for $\Re(z)\in\lbrb{0,1}$,
\begin{equation}\label{eq:LT}
    \begin{split}
     &\IntOI e^{-qt}\Ebb{\IPsi^{-z}(t)}dt=\frac{\MPsi(1-z)}{q}=\frac{h^{z}(q)}{\kappa_+(q,0)}\frac{\Gamma(1-z)}{\WkapP(q,1-z)}\WkapN(q,z),\\
     &\IntOI e^{-qt}\Ebb{\IPsih^{1-z}(t)}dt=\frac{\MPsih(z)}{q}=\frac{h^{1-z}(q)}{\kappa_-(q,0)}\frac{\Gamma(z)}{\WkapN(q,z)}\WkapP(q,1-z),
    \end{split}
\end{equation}
where we have used that $\hat{h}(q)=h(q)$, see \eqref{eq:h}, and $\hat{\kappa}_\pm=\kappa_{\mp}$.

Multiplying the last two expressions, we easily arrive at
\begin{equation}\label{eq:prodLT}
   \begin{split}
       &\frac{\MPsi(1-z)}{q}\frac{\Mcc_{\hat{\Psi}}(z)}{q}=\frac{h^2(q)}{h(q)\kappa_+(q,0)\kappa_-(q,0)}\Gamma(1-z)\Gamma(z)=\frac{h^2(q)}{q}\Gamma(1-z)\Gamma(z).
   \end{split}
\end{equation}
If $\xi$ is not a compound Poisson process, then $h(q)\equiv 1$. In that case, the left-hand side of \eqref{eq:prodLT} is the Laplace transform of the convolution, which appears on the left-hand side of \eqref{eq:convo}, and the right-hand side of \eqref{eq:prodLT} is simply the Laplace transform of the Heavy-side function, from where we obtain \eqref{eq:convo}. Otherwise, if $\Pi\lbrb{\Rb}=\lambda<\infty$ and $\Pi$ is continuous, then from \eqref{eq:h} we get the following
\begin{equation}\label{eq:h1}
	\begin{split}
		&h(q)=e^{-\int_{0}^{\infty}\lbrb{e^{-t}-e^{-qt}}e^{-\lambda t}\frac{dt}{t}}=e^{\ln\lbrb{\frac{\lambda+1}{\lambda+q}}}=\frac{\lambda+1}{\lambda+q},
	\end{split}
\end{equation}
where we have used that $\Pbb{\xi_t=0}=e^{-\lambda t}$ and Frullani's integral. Therefore, $h^2(q)$ is the Laplace transform of the function $\lbrb{\lambda+1}^2e^{-\lambda t}t\ind{t>0}$ and $h^2(q)/q$ is the Laplace transform of
\[\frac{(\lambda +1)^2}{\lambda^2}\lbrb{1-e^{-\lambda t}-\lambda te^{-\lambda t}}\ind{t>0}.\]
We thus conclude the proof.
\end{proof}

\begin{proof}[Proof of Corollary \ref{cor:convo}]
Relation \eqref{eq:convoS} is immediate from \eqref{eq:convo} once we use that $\IPsi(t)\stackrel{d}{=}\IPsih(t)$ when $\xi$ is symmetric. To prove \eqref{eq:convoSS}, we use \eqref{eq:convoS} for $z=1/2$ and note that in this case \eqref{eq:convoS} is simply the  convolution of $\Eb[\IPsi^{-1/2}(s)]$ with itself. Then \eqref{eq:LT} is invoked to get that
\[\frac{\MPsi\lbrb{\frac12}}{q}=\IntOI e^{-qt}\Ebb{\IPsi^{-\frac12}(t)}dt=\frac{\sqrt{\pi}}{\sqrt{q}}=\IntOI e^{-qt}t^{-\frac12}dt.\]
This establishes \eqref{eq:convoSS}. Relation \eqref{eq:convoS1} is immediate from \eqref{eq:convo1}, whereas \eqref{eq:convoSS1} follows in a similar fashion by setting $z=1/2$ in \eqref{eq:convoS1} and identifying the measures with the help of \eqref{eq:h1} above.
\end{proof}

\begin{proof}[Proof of Lemma \ref{lem:LTrep}]
Fix $q>0$ and $z\in\mathbb{C}$ such that $\Re(z)>0$. It is well-known, see \cite[(4.1)]{BarkerSavov2021} or the general expression in \cite[Theorem 2.4]{PatieSavov2018}, that
\begin{equation*}
    \begin{split}
       \frac{\Gamma(z)}{W_{\kappa}(q,z)} &=\Eb\Biggl[\lbrb{\int_{0}^{\mathbf{e_{\kappa(q,0)}}}e^{-V_s}ds}^{z-1}\Biggr],
    \end{split}
\end{equation*}
where $V=\lbrb{V_s}_{s\geq 0}$ is the unkilled subordinator pertaining to the Bernstein function $\kappa_q(z)=\kappa(q,z)-\kappa(q,0)$, and $\mathbf{e_{\kappa(q,0)}}$ is an independent of $V$ exponential random variable with parameter $\kappa(q,0)>0$. Next, we show that $\lbrb{V_s}_{s\leq \mathbf{e_{\kappa(q,0)}}}\stackrel{w}{=}\lbrb{H_s}_{s\leq L(\mathbf{e}_q)}$, where $(L^{-1},H)$ is the bivariate ladder subordinator, $L$ is the local time, and $\mathbf{e}_q$ is an independent of $(L^{-1},H)$ exponential random variable with parameter $q$. 
Pick $s_0=0<s_1<s_2<\cdots<s_n$ and $\theta_1,\ldots,\theta_n\in\mathbb{R}$. It follows that
 \begin{equation*}\label{pr:V=H}
     \begin{split}
     &\Ebb{e^{i\sum_{j=1}^n\theta_j\lbrb{V_{s_j}-V_{s_{j-1}}}}\ind{\mathbf{e_{\kappa(q,0)}}>s_n}}=\kappa(q,0)\IntOI e^{-\kappa(q,0)t}\Ebb{e^{i\sum_{j=1}^n\theta_j\lbrb{V_{s_j}-V_{s_{j-1}}}}\ind{t>s_n}} dt\\
     &\quad=\Ebb{e^{i\sum_{j=1}^n\theta_j\lbrb{V_{s_j}-V_{s_{j-1}}}}} e^{-\kappa(q,0)s_n}= e^{-\kappa(q,0)s_n}\prod_{j=1}^n e^{(s_j-s_{j-1})\lbrb{-\kappa(q,-i\theta_j)+\kappa(q,0)}}\\&\quad=\prod_{j=1}^n e^{-(s_j-s_{j-1})\kappa(q,-i\theta_j)},
     \end{split}
 \end{equation*}
and
 \begin{equation*}\label{pr:V=H1}
     \begin{split}
     &\Ebb{e^{i\sum_{j=1}^n\theta_j\lbrb{H_{s_j}-H_{s_{j-1}}}}\ind{L(\mathbf{e}_q)>s_n}}=q\IntOI e^{-qt}\Ebb{e^{i\sum_{j=1}^n\theta_j\lbrb{H_{s_j}-H_{s_{j-1}}}}\ind{L(t)>s_n}} dt\\
     &=\Ebb{e^{i\sum_{j=1}^n\theta_j\lbrb{H_{s_j}-H_{s_{j-1}}}} e^{-qL^{-1}(s_n)}}=\prod_{j=1}^n e^{-(s_j-s_{j-1})\kappa(q,-i\theta_j)};
     \end{split}
 \end{equation*}
thus, $\lbrb{V_s}_{s\leq \mathbf{e_{\kappa(q,0)}}}\stackrel{w}{=}\lbrb{H_s}_{s\leq L(\mathbf{e}_q)}$. Plugging this above, we get
\begin{equation*}
    \begin{split}
       \frac{\Gamma(z)}{W_{\kappa}(q,z)} &=\Eb\Biggl[\lbrb{\int_{0}^{L(\mathbf{e}_q)}e^{-H_s}ds}^{z-1}\Biggr].
    \end{split}
\end{equation*}
This concludes the proof.
\end{proof}

Next, we continue with the proof of Lemma \ref{lem:convo2}.

\begin{proof}[Proof of Lemma \ref{lem:convo2}]
Since $\xi$ is not a compound Poisson process, and thus $\kappa_+(q,0)\kappa_-(q,0)=q$ and $\kappa_\pm=\phi_\pm$, we note from \eqref{eq:recur} that, for any $q>0$ and $\Re(z)\in\lbrb{0,1}$,
\begin{equation*}
    \IntOI e^{-qt}\Ebb{\IPsi^{z-1}(t)}dt=\frac{1}{\kappa_+(q,0)}\frac{\Gamma(z)}{\WkapP(q,z)}\WkapN(q,1-z).
\end{equation*}
Rearranging the above equation, we get
\begin{equation*}
    \lbrb{\frac{1}{q}\frac{\Gamma(1-z)}{\WkapN(q,1-z)}}\IntOI e^{-qt}\Ebb{\IPsi^{z-1}(t)}dt=\frac{\Gamma(1-z)}{\kappa_+(q,0)}\lbrb{\frac{1}{q}\frac{\Gamma(z)}{\WkapP(q,z)}}.
\end{equation*}
Finally, employing \eqref{eq:LTrep} and using that $\kappa^{-1}_+$ is the Laplace transform of the potential measure $U_{L^{-1}_+}$, see \cite[Chapter III]{Ber96}, we conclude the proof by identifying the Laplace transforms of the two convolutions.
\end{proof}

\begin{proof}[Proof of Corollary \ref{cor:convo2}]
The proof follows from Lemma \ref{lem:convo2} and \eqref{eq:convo2}. First, observe that since $\xi$ is not a compound Poisson Process, then $h(q)\equiv 1$. Second, by symmetry, $\kappa_\pm(q,0)=\sqrt{q}$, hence $U_{L^{1}_\pm}(ds)=ds/(\sqrt{s}\Gamma(\lbrb{1/2}))$. Plugging all this information in \eqref{eq:convo2}, we arrive at \eqref{eq:convo2_1}.
\end{proof}~ 

We proceed with the proof of Theorem \ref{thm:NegMom}, starting from some preliminary results. Recall the quantities in \eqref{eq:aphi},
 $       \mathbf{a}_\phi=\inf\{u<0:\phi\in \mathbf{A}_{\lbrb{u,\infty}}\}\in[-\infty,0],\;
        \mathbf{u}_\phi=\inf\{u\in[\mathbf{a}_\phi,0]: \phi(u)=0\}\in[-\infty,0],\;
        \mathbf{\bar{a}}_\phi=\max\{ \mathbf{a}_\phi,\mathbf{u}_\phi\}\in[-\infty,0].$\\

\begin{lemma}\label{lem:kappa}
Let $\phi$ be any Bernstein function. Then, for any $z=\ab$ such that $a>\mathbf{a}_\phi$ and reals $x>y \geq 0$ with $x-y-a\neq 0$, it holds that
\begin{equation}\label{eq:kappa}
\begin{split}
    \abs{\phi(x)-\phi\lbrb{z+y}}&\leq \frac{\abs{x-y-z}}{\abs{x-y-a}}\abs{\phi(x)-\phi(a+y)}\\
    &\leq \abs{x-y-z}\phi'(\min\curly{x,a+y}),
    \end{split}
\end{equation}
and, if $x-y-a=0$, then
\begin{equation}\label{eq:kappa_2}
\begin{split}
      \abs{\phi(x)-\phi(x+ib)}\leq \int_0^{|b|} \abs{\phi'(x+iv)}dv\leq |b|\phi'(x).
    \end{split}
\end{equation}
Also, for $x>y>0$,
\begin{equation}\label{eq:kappa_1}
    (x-y)\phi'(x)\leq \phi(x)-\phi\lbrb{y}\leq (x-y)\phi'(y).
\end{equation}
\end{lemma}
\begin{proof}
Let $a\geq 0$. If $\Re(z)+y=x$, then \eqref{eq:kappa_2} follows from
\begin{equation*}
    \begin{split}
        \abs{\phi(x)-\phi(x+ib)}\leq \int_0^{|b|} \abs{\phi'(x+iv)}dv\leq |b|\phi'(x),
    \end{split}
\end{equation*}
where we have used that $\abs{\phi'(z)}\leq \phi'(\Re(z))$, see \cite[Proposition 5.2, (5.9)]{BarkerSavov2021}. If $x-y-a\neq 0$, then
\begin{equation*}
    \begin{split}
        &\abs{\phi(x)-\phi\lbrb{z+y}}=\abs{\int_{z+y\mapsto x}\phi'(w)dw}=\abs{1-\frac{z}{x-y}}\abs{\int_{0}^{x-y}\phi'\lbrb{z+y+\lbrb{1-\frac{z}{x-y}}s}ds}\\
        &\leq \abs{1-\frac{z}{x-y}}\abs{\int_{0}^{x-y}\phi'\lbrb{a+y+\lbrb{1-\frac{a}{x-y}}s}ds}=\frac{\abs{x-y-z}}{\abs{x-y-a}}\abs{\phi(x)-\phi(a+y)},
    \end{split}
\end{equation*}
where we have used in the inequality that $\abs{\phi'(z)}\leq \phi'(\Re(z))$. Thus, the first part of \eqref{eq:kappa} follows, and the second one uses the fact that $\phi'$ is non-decreasing, see \cite[Proposition 3.1 (2)]{BarkerSavov2021}. Relation \eqref{eq:kappa_1} is immediate.
\end{proof}

The next lemma also comes from estimates originally loosely made in \cite{BarkerSavov2021}.\\

\begin{lemma}\label{lem:boundW}
Let $\Wp$ be a Bernstein-Gamma function. Then $\Wp^{-1}\in\mathbf{A}_{(\mathbf{a}_\phi,\infty)}$ and, for any $\Re(z)>\max\{-2+\mathbf{\bar{a}}_\phi,-1+\mathbf{a}_\phi\}$,
\begin{equation}\label{eq:boundW_1}
    \abs{\frac{1}{\Wp(z+1)}}\leq C\frac{\abs{\phi\lbrb{z+1}\sqrt{|\phi\lbrb{z+2}|}}}{\sqrt{\phi(1)}}e^{-\int_{1}^{2+\Re(z)}\ln\phi(x)dx},
\end{equation}
where $C>0$ is universal constant in $\phi$ and $z$ .
\end{lemma}
\begin{proof}
The statement $\Wp^{-1}\in\mathbf{A}_{(\mathbf{a}_\phi,\infty)}$ comes from \cite[Theorem 4.1, (1)]{PatieSavov2018}, which shows that $\Wp$ is zero-free and meromorphic in $\Re(z)>\mathbf{a}_\phi$. The proof of \eqref{eq:boundW_1} then follows the lines of \cite[Lemma 4.4, p. 491]{BarkerSavov2021}, without using any estimates for the factors containing $\phi$, and noting that $\ln\phi(x)$ is well-defined on $x>\mathbf{\bar{a}}_\phi$, since $\mathbf{\bar{a}}_\phi$ controls the first zero or the loss of analiticity of $\phi$. The authors have generously applied their own asymptotic expansion of $\Wp$ (see \cite[Theorem 2.9, (2.15)]{BarkerSavov2021}) by omitting the square root in the denominator, see $\sqrt{\phi(1)}$ above. This square root is absolutely crucial for our estimates.
\end{proof}

The next estimate is trivial, but general enough to state it separately.\\

\begin{lemma}\label{lem:product}
Let $\phi$ be a non-constant Bernstein function and $z=\ab$. If $a>-1-\mathbf{a}_\phi$ and $a\notin\curly{1,2,\cdots,k-1}$, there exists $K:=K(z)>0$ such that
\begin{equation}\label{eq:product}
    \abs{\frac{\prod_{j=1}^{k-1}\lbrb{\phi(z+j)-\phi(k)}}{\prod_{j=1}^{k-1}\lbrb{\phi(j)-\phi(k)}}}\leq K\abs{\frac{\prod_{j=1}^{k-1}\lbrb{\phi(k)-\phi(a+j)}}{\prod_{j=1}^{k-1}\lbrb{\phi(k)-\phi(j)}}}.
\end{equation}
Moreover, $K(z)$ is locally bounded on $\{z:\Re(z)>-1+\mathbf{a}_\phi\}\setminus\{1,2,\cdots,k-1\}$. If $a\in\curly{1,2,,\cdots,k-1}$, there is $K>0$ such that
\begin{equation}\label{eq:product1}
\begin{split}
    \abs{\frac{\prod_{j=1}^{k-1}\lbrb{\phi(z+j)-\phi(k)}}{\prod_{j=1}^{k-1}\lbrb{\phi(j)-\phi(k)}}}&\leq K\frac{|b|\phi'(k)}{\phi(k)-\phi(k-a)}\abs{\frac{\prod_{j=1,j+a \neq k}^{k-1}\phi(k)-\phi(j+a)}{\prod_{j=1,j+a \neq k}^{k-1}\lbrb{\phi(k)-\phi(j)}}}.
    \end{split}
\end{equation}
\end{lemma}
\begin{proof}
From relation \eqref{eq:kappa} of Lemma \ref{lem:kappa}, we get that, when $a\notin\curly{1,2,\cdots,k-1}$,
\begin{equation*}
    \begin{split}
        \abs{\frac{\prod_{j=1}^{k-1}\lbrb{\phi(z+j)-\phi(k)}}{\prod_{j=1}^{k-1}\lbrb{\phi(j)-\phi(k)}}}\leq \prod_{j=1}^{k-1}\abs{1-\frac{ib}{k-j-a}} \abs{\frac{\prod_{j=1}^{k-1}\lbrb{\phi(k)-\phi(j+a)}}{\prod_{j=1}^{k-1}\lbrb{\phi(k)-\phi(j)}}}.
    \end{split}
\end{equation*}
Clearly, the first product is locally uniformly convergent, and so \eqref{eq:product} follows. By taking out the term with $j=k-a$, relation \eqref{eq:product1} is a consequence of \eqref{eq:kappa_2}.
\end{proof}

The next lemma summarizes and expands Lemma 4.3 in \cite{BarkerSavov2021} in a more convenient form.\\

\begin{lemma}\label{lem:derPhi}
Let $\phi$ be a non-constant Bernstein function. Then, for any $c>0$ and $A>1$, there exists $C_{\phi,c,A}\in\lbbrb{1,\infty}$, increasing in $c$ for fixed $A$ and decreasing in $A$ for fixed $c$, such that
\begin{equation}\label{eq:derPhi}
    \frac{\phi'(x)}{\phi'(x+c)}\leq C_{\phi,c,A}, \quad \text{ for all $x>A$.}
\end{equation}
Moreover, $C^*_{\phi,c}:=\limi{A}C_{\phi,c,A}=\liminfi{x}\phi'(x)/\phi'(x+c)\in\lbbrb{1,\infty}$. If, for some $y_0>0$,
\[\liminfi{x}\frac{\phi'(x)}{\phi'(x+y_0)}:=\beta(y_0)>1,\]
then, for any $y>0$, $\liminfi{x}\phi'(x)/\phi'(x+y)=\beta(y)>1$ and $\phi(\infty)<\infty$, that is, the underlying subordinator is in fact a compound Poisson process and there exists $\gamma>0$ such that $\lim_{x\rightarrow\infty}e^{\gamma x}\phi'(x)= 0$.
\end{lemma}
\begin{proof}
The first part is a direct reformulation of Lemma 4.3 in \cite{BarkerSavov2021}. Assume that  $\underline{\lim}_{x\rightarrow\infty}\phi'(x)/\phi'(x+y)=\beta(y)=\beta>1$, for some $y>0$. From the monotonicity of $\phi'$, it is clear that this holds for any $y>0$. Therefore, for any $\delta<1-\beta^{-1}$, there is $x_0>0$ such that, for any $x\geq x_0$, $\phi'(x+y)\leq \beta_\delta\phi'(x),$ where $\beta_\delta=\beta^{-1}+\delta<1$. Recursively, we get  $\phi'(x_0+ny)\leq \beta^n_\delta\phi'(x_0)$. From this, together with the fact that $\phi$ is non-increasing, see \cite[Proposition 3.1, (2)]{BarkerSavov2021}, we first conclude that
\begin{equation*}
    \phi(\infty)-\phi(x_0)=\int_{x_0}^\infty \phi'(x)dx\leq \sum_{n\geq 1}\beta^n_\delta\phi'(x_0)=\frac{\beta_\delta}{1-\beta_\delta}\phi'(x_0)<\infty,
\end{equation*}
and hence $\phi(\infty)<\infty$, which is equivalent to the underlying subordinator having finite activity, see \cite[Chapter I]{Ber96}. Second, for any $x\in\lbbrbb{x_0+ny,x_0+(n+1)y}$, we get that
\[\phi'(x)\leq \phi'(x_0+ny)\leq \beta^n_\delta\phi'(x_0),\]
from where we conclude that there is $\gamma>0$ such that $\lim_{x\rightarrow\infty}e^{\gamma x}\phi'(x)=0$.
\end{proof}

The next lemma clarifies \eqref{eq:product} in a rather general case.\\

\begin{lemma}\label{lem:estimate}
Let $\phi$ be a Bernstein function such that $\limi{x}\phi''(x)/\phi'(x)=0$. Then, for any $1>\epsilon>0$ and $z=\ab$ with $a\geq 0$, we have that, for all $k$ large enough,
\begin{equation}\label{eq:estimate}
     \abs{\frac{\prod_{j=1}^{k-1}\lbrb{\phi(z+j)-\phi(k)}}{\prod_{j=1}^{k-1}\lbrb{\phi(j)-\phi(k)}}}\leq K\lbrb{\frac{\phi'(k)}{\phi(k)}}^{a(1-\epsilon)},
\end{equation}
for some $K=K(z)>0$, which is locally bounded on $\{z:\Re(z)>-1+\mathbf{a}_\phi\}\setminus\{1,2,\cdots\}$.
If instead $a\in(\mathbf{a}_\phi-1,0)$, then, for any $1>\epsilon>0$, we have that, for all $k$ large enough,
\begin{equation}\label{eq:estimate1}
     \abs{\frac{\prod_{j=1}^{k-1}\lbrb{\phi(z+j)-\phi(k)}}{\prod_{j=1}^{k-1}\lbrb{\phi(j)-\phi(k)}}}\leq K\lbrb{\frac{\phi'(k)}{\phi(k)}}^{a(1+\epsilon)},
\end{equation}
for some $K=K(z)>0$, which is locally bounded  on $\{z:\Re(z)>-1+\mathbf{a}_\phi\}\setminus\{1,2,\cdots\}$.
\end{lemma}
\begin{proof}
Let $a\geq0$ be non-integer. We use the inequality \eqref{eq:product} to replace $z$ with $a=\Re(z)$. The properties of $K$ stem from this inequality. The constants that appear hereafter are explicit and are locally bounded in $a$. Then, for any fixed $j\geq 1$, we have
\begin{equation*}
    \begin{split}
        \limi{k}\abs{\frac{\phi(k)-\phi(a+j)}{{\phi(k)-\phi(j)}}}\in\lbrb{0,\infty}.
    \end{split}
\end{equation*}
Therefore, hereafter we can always work with $k\geq j_0$, for fixed $j_0\geq 1$, and replace the first $j_0$ terms by constants.
  Next, note that
\begin{equation}\label{eq:elemBound}
    \begin{split}
        \abs{\frac{\phi(k)-\phi(a+j)}{{\phi(k)-\phi(j)}}}&=1-\frac{\phi(a+j)-\phi(j)}{\phi(k)-\phi(j)}.
    \end{split}
\end{equation}
Then, regardless of $a>0$, setting $l>a$, we get from \eqref{eq:elemBound} that,  for $j\leq k-l$,
\begin{equation*}
\begin{split}
   \abs{\frac{\phi(k)-\phi(a+j)}{{\phi(k)-\phi(j)}}}=1-\frac{\phi(a+j)-\phi(j)}{\phi(k)-\phi(j)}>0,
\end{split}
\end{equation*}
since $\phi(j+a)<\phi(k)$. For the terms with $k-1\geq j>k-l$, using that $\phi'$ is non-increasing (see \cite[Proposition 3.1, (2)]{BarkerSavov2021}), we can check that
\[\frac{\phi(a+j)-\phi(j)}{\phi(k)-\phi(j)}\leq a\frac{\phi'(k-l)}{\phi'(k)},\]
and from Lemma \ref{lem:derPhi} these are bounded for all large enough $k$. At the cost of a constant locally bounded in $a>0$, we can dispense of these terms. Then, using in \eqref{eq:elemBound} the standard inequality $\ln(1-x)\leq -x$, for $x\in\lbrb{0,1}$, we get
\begingroup\allowdisplaybreaks
\begin{align}
   \prod_{j=j_0}^{k-l}\abs{\frac{\phi(k)-\phi(a+j)}{{\phi(k)-\phi(j)}}} &\leq \exp\curly{-\sum_{j=j_0}^{k-l}\frac{\phi(a+j)-\phi(j)}{\phi(k)-\phi(j)}} \nonumber\\
   &=\exp\curly{-a\sum_{j=j_0}^{k-l}\frac{\phi'(j)}{\phi(k)-\phi(j)}-\sum_{j=j_0}^{k-l}\frac{\int_{j}^{j+a}\int_{j}^x\phi''(y)dydx}{\phi(k)-\phi(j)}} \nonumber\\
   &\leq \exp\curly{-a\sum_{j=j_0}^{k-l}\frac{\phi'(j)}{\phi(k)-\phi(j)}-\frac{a^2}{2}\sum_{j=j_0}^{k-l}\frac{\phi''(j)}{\phi(k)-\phi(j)}} \nonumber\\
   &\leq \exp\curly{-a\sum_{j=j_0}^{k-l}\frac{\phi'(j-1)}{\phi(k)-\phi(j)}-a\sum_{j=j_0}^{k-l}\frac{\phi''(j-1)}{\phi(k)-\phi(j)}-\frac{a^2}{2}\sum_{j=j_0}^{k-l}\frac{\phi''(j)}{\phi(k)-\phi(j)}} \nonumber\\
   &\leq \exp\curly{-a\sum_{j=j_0}^{k-l}\frac{\phi'(j-1)}{\phi(k)-\phi(j)}-\lbrb{a+\frac{a^2}{2}}\sum_{j=j_0}^{k-l}\frac{\phi''(j-1)}{\phi(k)-\phi(j)}}, \label{pr:bound:mainTerm}
    \end{align}
\endgroup
where in the last three inequalities we have used that $-\phi''$ is positive and non-increasing, which comes from $\phi'$ being completely monotone (see \cite[Proposition 3.1, (2)]{BarkerSavov2021}) and $\phi'(j)\geq \phi'(j-1)+\phi''(j-1)$. Estimating the first term in \eqref{pr:bound:mainTerm} by the integral
\[-a\sum_{j=j_0}^{k-l}\frac{\phi'(j-1)}{\phi(k)-\phi(j)}\leq -a\int_{j_0-1}^{k-l}\frac{\phi'(x)}{\phi(k)-\phi(x)}dx,\]
we get with the help of Lemma \ref{lem:derPhi} in the last inequality that, for all $k$ large enough,
\begingroup\allowdisplaybreaks
\begin{align}
    \prod_{j=j_0}^{k-l}&\abs{\frac{\phi(k)-\phi(a+j)}{{\phi(k)-\phi(j)}}}\leq \lbrb{\frac{\phi(k)-\phi(k-l)}{\phi(k)-\phi(j_0-1)}}^{a}\exp\curly{-\lbrb{a+\frac{a^2}{2}}\sum_{j=j_0}^{k-l}\frac{\phi''(j-1)}{\phi(k)-\phi(j)}} \nonumber\\
    &\hspace{5em}\leq \lbrb{l\frac{\phi'(k-l)}{{\phi(k)-\phi(j_0-1)}}}^{a}\exp\curly{-\lbrb{a+\frac{a^2}{2}}\sum_{j=j_0}^{k-l}\frac{\phi''(j-1)}{\phi(k)-\phi(j)}} \nonumber\\
    &\hspace{5em}\leq \lbrb{2C_{\phi,l}l}^a\lbrb{\frac{\phi'(k)}{{\phi(k)-\phi(j_0-1)}}}^{a}\exp\curly{-\lbrb{a+\frac{a^2}{2}}\sum_{j=j_0}^{k-l}\frac{\phi''(j-1)}{\phi(k)-\phi(j)}}. \label{pr:bound:mainTerm1}
    \end{align}
\endgroup
If $\lim_{x\rightarrow\infty}\phi''(x)/\phi'(x)=0$, then since $-\phi''$ is positive and non-increasing,
\[-\phi''(j-1)=\so{\phi'(j-1)},\]
and we conclude \eqref{eq:estimate} for any $\epsilon>0$ by estimating the very last exponent of \eqref{pr:bound:mainTerm1} with the same integral, which we have used in the steps prior to \eqref{pr:bound:mainTerm1}. We sketch briefly the proof for $a\in(\mathbf{a}_\phi-1,0)$. We can again work with indices larger than some fixed $j_0$. In this case, since
\begin{equation*}\label{eq:elemBound1}
    \begin{split}
        \abs{\frac{\phi(k)-\phi(a+j)}{{\phi(k)-\phi(j)}}}&=1+\frac{\phi(j)-\phi(a+j)}{\phi(k)-\phi(j)},
    \end{split}
\end{equation*}
we do not need to discard terms with $j\geq k-l$, yet we shall keep this truncation. Then, using that $\ln(1+x)\leq x$, for $x\in(0,1)$, and $\phi'(j)\geq \phi'(j-1)+\phi''(j-1)$,
\begin{equation*}\label{pr:bound:mainTerm2}
\begin{split}
   \prod_{j=j_0}^{k-l}\biggl|\frac{\phi(k)-\phi(a+j)}{{\phi(k)-\phi(j)}}\biggr|&\leq \exp\curly{\sum_{j=j_0}^{k-l}\frac{\phi(j)-\phi(a+j)}{\phi(k)-\phi(j)}}\\
   &=\exp\curly{a\sum_{j=j_0}^{k-l}\frac{\phi'(j)}{\phi(k)-\phi(j)}-\sum_{j=j_0}^{k-l}\frac{\int_{j+a}^{j}\int_{x}^j\phi''(y)dydx}{\phi(k)-\phi(j)}}\\
   &\leq \exp\curly{a\sum_{j=j_0}^{k-l}\frac{\phi'(j)}{\phi(k)-\phi(j)}-\frac{a^2}{2}\sum_{j=j_0}^{k-l}\frac{\phi''(j+a)}{\phi(k)-\phi(j)}}\\
   &\leq \exp\curly{a\sum_{j=j_0}^{k-l}\frac{\phi'(j-1)}{\phi(k)-\phi(j)}+a\sum_{j=j_0}^{k-l}\frac{\phi''(j-1)}{\phi(k)-\phi(j)}-\frac{a^2}{2}\sum_{j=j_0}^{k-l}\frac{\phi''(j+a)}{\phi(k)-\phi(j)}}.
    \end{split}
\end{equation*}
Now, from Lemma \ref{lem:derPhi} and the assumptions, we get that
\[\phi''(j+a)=\so{\phi'(j)},\qquad \phi''(j-1)=\so{\phi'(j)},\]
and the proof follows as in the case for $a>0$.

The case where $a$ is positive integer is treated similarly by applying \eqref{eq:product1} to \eqref{pr:bound:mainTerm} and working in the same way with the remaining terms as in \eqref{pr:bound:mainTerm}, where after passing at the exponential level in \eqref{pr:bound:mainTerm}, we can add a term with index $j=k-a$ at the expense of a constant.
\end{proof}

The next result improves on the previous lemma. Since $\phi$ is increasing, we shall use $\varphi$ for its inverse.\\

\begin{lemma}\label{lem:estimate1}
Let $\phi$ be a Bernstein function such that $\limi{x}\phi''(x)/\phi'(x)=0$. Let $z=\ab$. If $a\geq 0$, then, for all $k$ large enough, it holds that
\begin{equation}\label{eq:estimate_1_1}
     \abs{\frac{\prod_{j=1}^{k-1}\lbrb{\phi(z+j)-\phi(k)}}{\prod_{j=1}^{k-1}\lbrb{\phi(j)-\phi(k)}}}\leq K\lbrb{\frac{\phi'(k)}{{\phi(k)}}}^{a}e^{8C\lbrb{a+\frac{a^2}{2}}\frac{\phi(k)}{k\phi'(k)}\lbrb{\frac{\ln(k)}{k}+\frac{1}{\varphi(\phi(k)/2)}}},
\end{equation}
for some $K:=K(z)>0$, which is locally bounded on $\{z:\Re(z)\geq 0\}\setminus\{1,2,\cdots\}$, and constant $C>0$. If  $\limsupi{x}\phi'(x)/\phi'(2x)<\infty$, we have that, for $a\geq 0$,
\begin{equation}\label{eq:estimate_1_2}
    \abs{\frac{\prod_{j=1}^{k-1}\lbrb{\phi(z+j)-\phi(k)}}{\prod_{j=1}^{k-1}\lbrb{\phi(j)-\phi(k)}}}\leq K\lbrb{\frac{\phi'(k)}{{\phi(k)}}}^{a}e^{C\lbrb{a+\frac{a^2}{2}}\phi(k)\lbrb{\frac{8c\ln{k}}{k^2\phi'(k)}+\frac{\varphi'(y_k)}{\varphi^2(y_k)}}},
\end{equation}
for some $K:=K(z)>0$, which is locally bounded on $\{z:\Re(z)\geq 0\}\setminus\{1,2,\cdots\}$, and constants $C>0$ and $c>0$ are constants, where $y_k\in\lbbrbb{\phi(k)/2,\phi(k/2)}$ is a solution of the equation
\[\phi(k)=\phi(x)+x\phi'(x).\]
When $a\in(-1+\mathbf{a}_\phi,0)$, the claims hold with $-a+a^2/2$ in the exponents above and a possibly different constant $C$, which in this case may depend on $a$.
\end{lemma}~ 

\begin{remark}\label{rem:estimate1}
We note that $\overline{\lim}_{x\rightarrow\infty}\phi'(x)/\phi'(2x)<\infty$ implies that $\lim_{x\rightarrow\infty}\phi''(x)/\phi'(x)=0$. Indeed, from the proof below,
\[\phi'(x)\geq -\int_{x}^{2x}\phi''(v)dv\geq -x\phi''(2x)>0.\]
Therefore, from the assumption, we have that $\overline{\lim}_{x\rightarrow\infty}x\phi''(2x)/\phi'(2x)<\infty$.
\end{remark}

\begin{proof}[Proof of Lemma \ref{lem:estimate1}]
For the case $a>0$, we use the inequality \eqref{pr:bound:mainTerm1} with $u_a=a+a^2/2$ and some fixed $l$ such that $l>a$. We work with $k>j_0+2l$. For $j_0\leq j< k-l$, we have the chain of inequalities
\begin{equation*}
    \begin{split}
        \phi(j)-\phi(j-2)\leq 2\phi'(j-2)\leq 2 C_{\phi,2,j_0-2}\phi'(j),
    \end{split}
\end{equation*}
where we have used that $\phi'$ is non-increasing (see \cite[Proposition 3.1, (2)]{BarkerSavov2021}) and \eqref{eq:derPhi}. Therefore, using $\phi(k)-\phi(j)\geq (k-j)\phi'(j)$, we get for some $C>0$ that
\begin{equation*}
    \begin{split}
        \phi(k)-\phi(j-2)&\leq \phi(k)-\phi(j)+2\phi'(j-2)\leq \phi(k)-\phi(j)+2 C_{\phi,2,j_0-2}\phi'(j)\\
        &\leq \phi(k)-\phi(j)+2 C_{\phi,2,j_0-2}\frac{\phi(k)-\phi(j)}{k-j}\leq \lbrb{\phi(k)-\phi(j)}\lbrb{1+\frac{2 C_{\phi,2,j_0-2}}{l}}\\
        &=C\lbrb{\phi(k)-\phi(j)}.
    \end{split}
\end{equation*}
We employ this fact and that $-\phi''$ is positive and non-increasing to estimate the exponent in \eqref{pr:bound:mainTerm1} for $j_0>2$ as follows
\begin{equation}\label{pr:estimate1a}
    \begin{split}
        -u_a\sum_{j=j_0}^{k-l}\frac{\phi''(j-1)}{\phi(k)-\phi(j)}&\leq Cu_a\int_{j_0-2}^{k-l}\frac{-\phi''(x)}{\phi(k)-\phi(x)}dx.
    \end{split}
\end{equation}
Next, recalling that $\varphi=\phi^{-1}$, we set $b_k=\varphi\lbrb{\phi(k)/2}$. Using the fact that $2\phi(x/2)\geq \phi(x)$, for $x>0$, see \cite[Chapter III (6)]{Ber96}, we can deduce that $b_k\leq k/2.$ We obtain \eqref{eq:estimate_1_1} by estimating \eqref{pr:estimate1a} in the following fashion
\begin{equation}\label{pr:intermediate}
\begin{split}
        -u_a\sum_{j=j_0}^{k-l}\frac{\phi''(j-1)}{\phi(k)-\phi(j)}&\leq Cu_a\int_{j_0-2}^{k-l}\frac{-\phi''(x)}{\phi(k)-\phi(x)}dx\\
        &\leq Cu_a\lbrb{\int_{j_0-2}^{b_k}\frac{1}{x^2}dx+\frac{2\phi(k)}{\phi'(k)}\int_{b_k}^{k-l}\frac{1}{x^2\lbrb{k-x}}dx}\\
        &\leq u_a\lbrb{\frac{4}{j_0-2}+\frac{2\phi(k)}{k^2\phi'(k)}\int_{\frac{b_k}{k}}^{1-\frac{1}{k}}\frac{1}{x^2(1-x)}dx},
    \end{split}
\end{equation}
where we have used in the first inequality that $(i)$ $-\phi''(x)\leq 2\phi(x)/x^2$ (see \cite[Proposition 3.1, (3.3)]{PatieSavov2018}), $(ii)$ $\phi$ is increasing, and $(iii)$ $\phi(k)-\phi(x)\geq (k-x)\phi'(k)$ as $\phi'$ is decreasing (see \cite[Proposition 3.1, (2)]{BarkerSavov2021}). Since $b_k\leq k/2$, we can split above integral at $1/2$ and estimate the bounded terms to get
\begin{equation*}
    \begin{split}
    \frac{\phi(k)}{k^2\phi'(k)}\int_{\frac{b_k}{k}}^{1-\frac{1}{k}}\frac{1}{x^2(1-x)}dx
        &\leq 8\frac{\phi(k)}{k\phi'(k)}\lbrb{\frac{1}{\varphi(\phi(k)/2)}+\frac{\ln(k)}{k}}.
    \end{split}
\end{equation*}
Plugging this in \eqref{pr:estimate1a} concludes the proof of \eqref{eq:estimate_1_1} by an application of \eqref{pr:bound:mainTerm1}. To prove \eqref{eq:estimate_1_2}, we deviate from \eqref{pr:intermediate} by estimating the integral in \eqref{pr:estimate1a} in a different way. First, we note from the monotonicity of $-\phi''$ that
\[\phi'(x)=-\int_x^\infty \phi''(v)dv\geq -\int_{x}^{2x}\phi''(v)dv\geq -x\phi''(2x).\]
Since $\overline{\lim}_{x\rightarrow\infty}\phi'(x)/\phi'(2x)<\infty$, there is a constant $c$ such that, for $x>(j_0-1)/2$, we have $-\phi''(x)\leq c\phi'(x)/x$. Substituting this in the integral in \eqref{pr:intermediate}, we obtain
\begin{equation*}\label{pr:intermediate1}
    \begin{split}
        -u_a\sum_{j=j_0}^{k-l}\frac{\phi''(j-1)}{\phi(k)-\phi(j)}
        &\leq Cu_a\lbrb{\frac{4}{j_0-2}+c\int_{b_k}^{k-l}\frac{\phi'(x)}{x\lbrb{\phi(k)-\phi(x)}}dx}\\
        &= Cu_a\lbrb{\frac{4}{j_0-2}+c\int_{b_k}^{\frac{k}2}\frac{\phi'(x)}{x\lbrb{\phi(k)-\phi(x)}}dx+c\int_{\frac{k}2}^{k-l}\frac{\phi'(x)}{x\lbrb{\phi(k)-\phi(x)}}dx}.
    \end{split}
\end{equation*}
Recall that $\varphi=\phi^{-1}.$ Then
\begin{equation*}
    \begin{split}
        \int_{b_k}^{\frac{k}2}\frac{\phi'(x)}{x\lbrb{\phi(k)-\phi(x)}}dx&= \int_{\frac{\phi(k)}{2}}^{\phi\lbrb{\frac{k}{2}}}\frac{1}{\varphi(y)\lbrb{\phi(k)-y}}dy.
    \end{split}
\end{equation*}
We note that, letting $f_k(y):=\varphi(y)\lbrb{\phi(k)-y}, f_k'(y)=\varphi'(y)\lbrb{\phi(k)-y}-\varphi(y)$, we have
\begin{equation*}
    \begin{split}
        f'_k\lbrb{\frac{\phi(k)}{2}}&=\varphi'\lbrb{\frac{\phi(k)}{2}}\frac{\phi(k)}{2}-\varphi\lbrb{\frac{\phi(k)}{2}}\geq 0;\\
        f'_k\lbrb{\phi\lbrb{\frac{k}{2}}}&=\frac{\phi(k)-\phi\lbrb{\frac{k}{2}}-\frac{k}{2}\phi'\lbrb{\frac{k}{2}}}{\phi'\Bigl(\frac{k}{2}\Bigr)}\leq 0,
    \end{split}
\end{equation*}
where the first inequality follows from $\varphi'(x)x\geq \varphi(x)$, which comes from the observation  that $\phi'(x)\leq \phi(x)/x$, and the second one follows from the fact that $\phi'$ is decreasing. Therefore, there is $y_k\in[\phi(k)/2,\phi(k/2)]$ that minimizes $f_k$ on this region, $\phi(k)-y_k=\varphi(y_k)/\varphi'(y_k)$ and $f_k(y_k)=\varphi^2(y_k)/\varphi'(y_k)$. From these, we also get that $\varphi(y_k)/\varphi'(y_k)\geq \phi(k)-\phi(k/2)$. Hence, we obtain that
\begin{equation*}
    \begin{split}
        \int_{b_k}^{\frac{k}2}\frac{\phi'(x)}{x\lbrb{\phi(k)-\phi(x)}}dx&\leq \frac{\varphi'(y_k)}{\varphi^2(y_k)}\lbrb{\phi\lbrb{\frac{k}{2}}-\frac{\phi(k)}{2}}\leq \phi(k)\frac{\varphi'(y_k)}{\varphi^2(y_k)}. 
    \end{split}
\end{equation*}
Recall that $\phi'(x)\leq \phi(x)/x$ (see \cite[Proposition 3.1, (3.3)]{PatieSavov2018}) and that $\phi$ and $\phi'$ are  increasing and decreasing, respectively. We can then bound the remaining integral as follows
\begin{equation*}
    \begin{split}
        \int_{\frac{k}2}^{k-l}\frac{\phi'(x)}{x\lbrb{\phi(k)-\phi(x)}}dx&\leq  \int_{\frac{k}2}^{k-l}\frac{\phi(x)}{x^2\lbrb{\phi(k)-\phi(x)}}dx\leq\frac{\phi(k)}{\phi'(k)}\int_{\frac{k}2}^{k-l}\frac{1}{x^2\lbrb{k-x}}dx\leq \frac{4\phi(k)\ln{k}}{k^2\phi'(k)}.
    \end{split}
\end{equation*}
We have thus proved \eqref{eq:estimate_1_2}.  The case $a<0$ is treated similarly using \eqref{pr:bound:mainTerm1} instead. We sketch the arguments.  We pick $j_0>-2a$, $l>-a$ and $k>2l+j_0$. The first two terms in \eqref{pr:bound:mainTerm1} are treated exactly as above. The last one is taken care of as follows
\begin{equation}\label{pr:estimate1b}
    \begin{split}
        -u_a\sum_{j=j_0}^{k-l}\frac{\phi''(j+a)}{\phi(k)-\phi(j)}&\leq C'u_a\int_{j_0+a-1}^{k-l}\frac{-\phi''(x)}{\phi(k)-\phi(x)}dx,
    \end{split}
\end{equation}
where $C'$ is chosen precisely as in the arguments preceding \eqref{pr:estimate1b}, which requires us to estimate $\phi(k)-\phi(j)$ from below with $C'(\phi(k)-\phi(j+a-1))$. This is possible for fixed $a$.
\end{proof}

Next, for any $\phi$ and $z=\ab$, we set
\begin{equation}\label{eq:H}
    H_\phi(z,k):=\frac{\prod_{j=1}^k\lbrb{\phi(k)-\phi(z+j)}}{\prod_{j=1}^{k-1}\lbrb{\phi(k)-\phi(j)}}\frac{1}{\phi(k)} \frac{\Gamma(z+1)}{W_{\phi_{(k)}}(z+1)},
\end{equation}
where $\phi_{(k)}(x)=\phi(x+k)-\phi(k)$, $k\geq0$ is a Bernstein function. We provide some estimates for $H_\phi$.\\

\begin{lemma}\label{lem:H}
Let $\phi$ be a Bernstein function such that $\limi{x}\phi''(x)/\phi'(x)=0$. Then, for any $1>\epsilon>0$ and $z=\ab$ with $a\geq 0$, we have that, for all $k$ large enough,
\begin{equation*}\label{eq:estimateH}
     \abs{H_\phi(z,k)}\leq K\frac{\lbrb{\phi'(k)}^{1-a\epsilon}}{\lbrb{\phi(k)}^{1+a-a\epsilon}},
\end{equation*}
for some $K:=K(z)>0$, which is locally bounded on the region $\{z:\Re(z)>0\}\setminus\{1,2,\cdots\}$. If $a\in(\mathbf{a}_\phi\ind{\phi(0)=0}-1,0)$, then
\begin{equation}\label{eq:estimate1a}
     \abs{H_\phi(z,k)}\leq K\frac{\lbrb{\phi'(k)}^{1+a\epsilon}}{\lbrb{\phi(k)}^{1+a+a\epsilon}},
\end{equation}
for some $K:=K(z)$, which is locally bounded on $\{z:0\geq \Re(z)>-1+\mathbf{a}_\phi\ind{\phi(0)=0}\}$. Under the conditions of {\rm Lemma} \ref{lem:estimate1}, we have, for $a>-1+\mathbf{a}_\phi\ind{\phi(0)=0}$, the bounds
\begin{equation}\label{eq:estimate2}
\begin{split}
&\abs{H_\phi(z,k)}\leq K\frac{\phi'(k)}{\lbrb{\phi(k)}^{1+a}}e^{8C\lbrb{|a|+\frac{a^2}{2}}\frac{\phi(k)}{k\phi'(k)}\lbrb{\frac{\ln(k)}{k}+\frac{1}{\varphi(\phi(k)/2)}}},\\
&\abs{H_\phi(z,k)}\leq K\frac{\phi'(k)}{\lbrb{\phi(k)}^{1+a}}e^{C\lbrb{|a|+\frac{a^2}{2}}\phi(k)\lbrb{\frac{8c\ln{k}}{k^2\phi'(k)}+\frac{\varphi'(y_k)}{\varphi^2(y_k)}}},
\end{split}
\end{equation}
where the constants have the same properties as in {\rm Lemma} \ref{lem:estimate1} and $y_k\in\lbbrbb{\phi(k)/2,\phi(k/2)}$ is a solution of the equation $\phi(k)=\phi(x)+x\phi'(x)$.
\end{lemma}
\begin{proof}
The proof when $a\geq 0$ is a combination of \eqref{eq:estimate} and inequality \eqref{eq:boundW_1}, which yields, for all $k$ large enough,
 \begin{equation*}
     \begin{split}
          \abs{H_\phi(z,k)}&\leq K\abs{\phi(k)-\phi(k+a)}\frac{\lbrb{\phi'(k)}^{a-a\epsilon}}{\lbrb{\phi(k)}^{a(1-\epsilon)+1}}\frac{\abs{\lbrb{\phi(z+1+k)-\phi(k)}\sqrt{|\phi\lbrb{z+2+k}-\phi(k)|}}}{\sqrt{\phi(k+1)-\phi(k)}}\\
          &\times e^{-\int_{1}^{2+a}\ln\lbrb{\phi(x+k)-\phi(k)}dx},
     \end{split}
 \end{equation*}
where $\Gamma(z+1)$ is included in $K$. Applying first the second inequality in \eqref{eq:kappa} to the terms in the modulus, and then \eqref{eq:derPhi} and $\phi(k+1)-\phi(k)\geq \phi'(k+1)$ (see \cite[Proposition 3.1, (2)]{BarkerSavov2021}), we arrive at
 \begin{equation*}
     \begin{split}
          \abs{H_\phi(z,k)}&\leq K'\frac{\lbrb{\phi'(k)}^{2+a-a\epsilon}}{\lbrb{\phi(k)}^{a(1-\epsilon)+1}} e^{-\int_{1}^{2+a}\ln\lbrb{\phi(x+k)-\phi(k)}dx},
     \end{split}
 \end{equation*}
for some $K'=K(z)>0$. Since $\phi(x+k)-\phi(k)\geq x\phi'(2+a+k)$ (see \cite[Proposition 3.1, (2)]{BarkerSavov2021}), we get
 \[\ln\lbrb{\phi(x+k)-\phi(k)}\geq \ln(x)+\ln\lbrb{\phi'(2+a+k)}.\]
Now, using \eqref{eq:derPhi} to compare $\phi'(2+a+k)$ with $\phi'(k)$, we derive, at a cost of a constant,
  \begin{equation*}
     \begin{split}
          \abs{H_\phi(z,k)}&\leq K''\frac{\lbrb{\phi'(k)}^{1-a\epsilon}}{\lbrb{\phi(k)}^{a(1-\epsilon)+1}},
     \end{split}
 \end{equation*}
for some $K''=K(z)>0$. Let $a\in(\mathbf{a}_\phi\ind{\phi(0)=0}-1,0)$ and $z$ be not a negative integer.  Choose $l\geq 1$ such that $a+l\in\lbbrb{0,1}$. Using the recurrence relation for $\Wp$, we get that
 \begin{equation}\label{eq:Hr}
     \begin{split}
\abs{H_\phi(z,k)}&=\abs{\lbrb{\phi(k)-\phi(z+k)}\frac{\prod_{j=1}^{k-1}\lbrb{\phi(k)-\phi(z+j)}}{\prod_{j=1}^{k-1}\lbrb{\phi(k)-\phi(j)}}\frac{1}{\phi(k)} \frac{\Gamma(z+1+l)}{W_{\phi_{(k)}}(z+1+l)}\frac{\prod_{j=1}^l\phi_{(k)}(z+j)}{\prod_{j=1}^l z+j}}.
     \end{split}
 \end{equation}
Using the estimates in \eqref{eq:kappa}, \eqref{eq:boundW_1},  \eqref{eq:estimate1}, and similar arguments as above, we get
 \begin{equation*}
     \begin{split}
         \abs{H_\phi(z,k)}&\leq K\lbrb{\phi'(k)}^{l+1}\lbrb{\frac{\phi'(k)}{\phi(k)}}^{a(1+\epsilon)}\frac{1}{\phi(k)}\frac{\abs{\lbrb{\phi(z+1+k)-\phi(k)}\sqrt{|\phi\lbrb{z+2+k}-\phi(k)|}}}{\sqrt{\phi(k+1)-\phi(k)}}\\
          &\times e^{-\int_{1}^{2+l+a}\ln\lbrb{\phi(x+k)-\phi(k)}dx}\\
          &\leq K' \frac{\lbrb{\phi'(k)}^{a(1+\epsilon)+l+2}}{\lbrb{\phi(k)}^{a+a\epsilon+1}}e^{-\int_{1}^{2+l+a}\ln\lbrb{\phi(x+k)-\phi(k)}dx},
     \end{split}
 \end{equation*}
where $K,K'$ are positive constants independent of $k$. Similar to the above, we estimate $\phi(x+k)-\phi(k)\geq x\phi'(2+l+a+k)$ and use \eqref{eq:derPhi} to obtain
 \begin{equation*}
     \abs{H_\phi(z,k)}\leq K'' \frac{\lbrb{\phi'(k)}^{1+a\epsilon}}{\lbrb{\phi(k)}^{1+a+a\epsilon}},
 \end{equation*}
with $K''>0$. Finally, let $z=a>\mathbf{a}_\phi\ind{\phi(0)=0}-1$ be a negative integer when $\mathbf{a}_\phi\ind{\phi(0)=0}<0$. For any $k\geq 1$, $\phi_{(k)}(0)=0$ and $\mathbf{a}_{\phi_{(k)}}=\mathbf{a}_\phi+k$, see \eqref{eq:aphi}. According to \cite[Theorem 2.1, (2.10)]{PatieSavov2018}, the function $\Gamma(z)\bigl/W_{\phi_{(k)}}(z)$ is analytic on $\{z\in\Cb:\Re(z)>\mathbf{a}_\phi+k\}$. Choose $z'=a+w$ and $w\in\Cb$ such that $\Re(z')>-1-\mathbf{a}_\phi\ind{\phi(0)=0}$. Using an argument as in \eqref{eq:Hr}, we have
 \[\frac{\Gamma(z'+1)}{W_{\phi_{(k)}}(z'+1)}=\frac{\Gamma(z'+1-a)}{W_{\phi_{(k)}}(z'+1-a)}\frac{\prod_{j=1}^{-a}\phi_{(k)}(z'+j)}{\prod_{j=1}^{-a} z'+j}=\frac{\Gamma(w+1)}{W_{\phi_{(k)}}(w+1)}\frac{\prod_{j=1}^{-a}\phi_{(k)}(a+w+j)}{\prod_{j=1}^{-a} a+w+j},\]
and thus, we get
\begin{equation*}\label{pr:poles}
 \limo{w}\frac{\Gamma(a+w+1)}{W_{\phi_{(k)}}(a+w+1)}=\limo{w}\frac{\Gamma(w+1)}{W_{\phi_{(k)}}(w+1)}\frac{\prod_{j=1}^{-a}\phi_{(k)}(a+w+j)}{\prod_{j=1}^{-a} a+w+j}=\phi'(k)\prod_{j=1}^{-a-1}\frac{\phi_{(k)}(a+j)}{a+j},
 \end{equation*}
where we have used that $\phi'_{(k)}(0)=\phi'(k)$. Therefore, in \eqref{eq:H}, we get
 \begin{equation*}
     \begin{split}
         \abs{H_\phi(a,k)}&=\abs{\lbrb{\phi(k)-\phi(a+k)}\frac{\prod_{j=1}^{k-1}\lbrb{\phi(k)-\phi(a+j)}}{\prod_{j=1}^{k-1}\lbrb{\phi(k)-\phi(k)}}\frac{\phi'(k)}{\phi(k)} \prod_{j=1}^{-a-1}\frac{\phi_{(k)}(a+j)}{ a+j}}.
     \end{split}
 \end{equation*}
 Finally, from Lemma \ref{lem:kappa}, we deduce
 \[\abs{\phi'(k)\prod_{j=1}^{-a-1}\frac{\phi_{(k)}(a+j)}{ a+j}}\leq \lbrb{\phi'(k)}^{-a},\]
 which has the same contribution to the bound as above. Thus, we conclude the proof of \eqref{eq:estimate1a}. The bounds \eqref{eq:estimate2} are derived in  the same way by invoking Lemma \ref{lem:estimate1}, which gives better bounds for
 \[\abs{\frac{\prod_{j=1}^{k-1}\lbrb{\phi(z+j)-\phi(k)}}{\prod_{j=1}^{k-1}\lbrb{\phi(j)-\phi(k)}}}.\]
 The rest is the same. Therefore, we conclude the proof.
\end{proof}

We now prove Theorem \ref{thm:NegMom}.

\begin{proof}[Proof of Theorem \ref{thm:NegMom}]
We will use the proof of Theorem 2.14 in \cite{BarkerSavov2021}, since it employs the general representation of the moments of $I_\phi(t)$ as an infinite convolution (see \cite[Theorem 2.12]{BarkerSavov2021}), yielding, for $\Re(z)>0$,
\begin{equation}\label{pr:NegMom_base}
\frac{d}{dt}\Ebb{I^{z}_{\phi}(t)}=\sum_{k\geq 1}\frac{\prod_{j=1}^k\lbrb{\phi(z+j)-\phi(k)}}{\prod_{j=1}^{k-1}\lbrb{\phi(j)-\phi(k)}}e^{-\phi(k)t}\frac{\Gamma(z+1)}{W_{\phi_{(k)}}(z+1)}=\sum_{k\geq 1}\phi(k)H_\phi(z,k)e^{-\phi(k)t},
\end{equation}
provided that the series is absolutely convergent and $\phi(\infty)=\infty$. Then \eqref{eq:NegMom} follows from
\[\Ebb{I^{z}_{\phi}(t)}=\Ebb{I^{z}_{\phi}(\infty)}-\int_{t}^{\infty }\frac{d}{ds}\Ebb{I^{z}_{\phi}(s)}ds,\]
the integration of the exponents in \eqref{pr:NegMom_base}, and the fact that the classical identity $\Eb[I^{z}_{\phi}(\infty)]=\Gamma(z+1)/W_{\phi}(z+1)$ holds. Fix $z=\ab$ with $a>0$. We proceed to show that the series in \eqref{pr:NegMom_base} is absolutely convergent. First, note that since $\overline{\lim}_{x\rightarrow\infty}\phi'(x)/\phi'(2x)<\infty$, it follows from Remark \ref{rem:estimate1} that $\lim_{x\rightarrow\infty}\phi''(x)/\phi'(x)=0$ and all estimates in Lemma \ref{lem:H} are valid. Second, with $\varphi=\phi^{-1}$, we have
\[\infty=\limi{x}x^2\phi'(x)=\limi{y}\frac{\varphi^2(y)}{\varphi'(y)}.\]
Third, recalling the definition of $y_k$ in Lemma \ref{lem:H} and since $\phi(\infty)=\infty$, we see that
\[\infty=\lim{k}\phi(k)=\limi{k}\lbrb{\phi(y_k)+y_k\phi'(y_k)}.\]
Henceforth, $\lim_{k\rightarrow\infty}y_k=\infty$ and $\lim_{k\rightarrow\infty}\varphi'(y_k)/\varphi^2(y_k)=0$. As a result, for any $a>-1-\mathbf{a}_\phi$, the second bound in \eqref{eq:estimate2} implies that, for any $\varepsilon>0$ and all $k\geq k_\varepsilon$,
\begin{equation}\label{pr:H}
\abs{\phi(k)H_\phi(z,k)}\leq  K\frac{\phi'(k)}{\lbrb{\phi(k)}^{a}}e^{C\lbrb{a+\frac{a^2}{2}}\phi(k)\lbrb{\frac{8c\ln{k}}{k^2\phi'(k)}+\varepsilon}}.
\end{equation}
Set $\mathcal{K}_\varepsilon:=\bigl\{k\geq k_\varepsilon: \frac{8\ln{k}}{k^2\phi'(k)}\leq \varepsilon\bigr\}$. Using the bound \eqref{eq:estimate} with $\epsilon=\varepsilon$ on $\mathcal{K}^c_\varepsilon$, we get the unified bound
\begin{equation}\label{pr:H1}
\abs{\phi(k)H_\phi(z,k)}\leq  K\frac{\phi'(k)}{\lbrb{\phi(k)}^{a}}e^{2c\varepsilon\lbrb{a+\frac{a^2}{2}}\phi(k)}\ind{k\in\mathcal{K}_\varepsilon}+K\frac{\lbrb{\phi'(k)}^{1-a\varepsilon}}{\lbrb{\phi(k)}^{1+a-a\varepsilon}}\ind{k\in\mathcal{K}^c_\varepsilon}.
\end{equation}
Choosing $\varepsilon=\varepsilon(\delta,a,c,C)>0$ by estimating $\phi'(k)$ on $\mathcal{K}^c_\varepsilon$, we show that, for any $\delta>0$ small enough,
\begin{equation}\label{pr:H2}
\abs{\phi(k)H_\phi(z,k)}\leq  K\frac{\phi'(k)}{\lbrb{\phi(k)}^{a}}e^{\delta\phi(k)}+K\lbrb{\frac{1}{\varepsilon}\frac{\ln{k}}{k^2}}^{1-\delta}\frac{1}{\lbrb{\phi(k)}^{1+a-a\varepsilon}}.
\end{equation}
From \eqref{pr:NegMom_base} we get that, for $t>\delta$,
\begin{equation*}
    \begin{split}
    \sum_{k\geq k_\varepsilon}\abs{\phi(k)H_\phi(z,k)}e^{-\phi(k)t}\leq K\sum_{k\geq k_\varepsilon}\lbrb{\frac{\phi'(k)}{\lbrb{\phi(k)}^{a}}e^{-(t-\delta)\phi(k)}+\frac{K}{\varepsilon}\lbrb{\frac{\ln{k}}{k^2}}^{1-\delta}\frac{1}{\lbrb{\phi(k)}^{1+a-a\varepsilon}}}.
    \end{split}
\end{equation*}
The second series is obviously convergent, while for the first one we use the integral comparison, noting that as $\phi'$ is decreasing and $\phi$ is increasing,
\[\sum_{k\geq k_\varepsilon}\frac{\phi'(k)}{\lbrb{\phi(k)}^{a}}e^{-(t-\delta)\phi(k)}\leq \int_{k_\varepsilon-1}^{\infty} \frac{\phi'(x)}{\phi^a(x)}e^{-(t-\delta)\phi(x)}dx=\int_{\phi(k_\varepsilon-1)}^\infty x^{-a}e^{-(t-\delta)x}dx<\infty.\]
Since $\delta$ is arbitrary, we establish the absolute continuity of the series in \eqref{pr:NegMom_base} and, by the aforementioned reference to the proof in \cite{BarkerSavov2021}, we verify \eqref{eq:NegMom} for all $t>0$ and $z=\ab$ such that $a>0$. We emphasize that the very same estimates as in \eqref{pr:H}, \eqref{pr:H1} and \eqref{pr:H2} are valid when $a\in(-1-\mathbf{a}_\phi,0]$, with $a$ replaced by $-a$ in the exponent of \eqref{pr:H}. Therefore, the series in \eqref{pr:NegMom_base} is absolutely convergent. Thus, by analytic continuation, we conclude the proof of this theorem.
\end{proof}

\begin{proof}[Proof of Corollary \ref{cor:NegMom}]
We sketch only the case $l=-1$; the rest is similar. We apply \eqref{eq:NegMom}, for $z=-1+\varepsilon$ with $\varepsilon>0$, to get
\begin{equation*}
\Ebb{I^{-1+\varepsilon}_{\phi}(t)}=\frac{\Gamma(\varepsilon)}{W_{\phi}(\varepsilon)}-\sum_{k\geq 1}\frac{\prod_{j=1}^k\lbrb{\phi(-1+\varepsilon+j)-\phi(k)}}{\prod_{j=1}^{k-1}\lbrb{\phi(j)-\phi(k)}}\frac{e^{-\phi(k)t}}{\phi(k)} \frac{\Gamma(\varepsilon)}{W_{\phi_{(k)}}(\varepsilon)}.
\end{equation*}
Since $\phi'(0)<\infty$ and $\phi'_{(k)}(0)=\phi'(k)$, $k\geq 1$, the recurrent relation \eqref{eq:recur} yields
\begin{equation*}
\Ebb{I^{-1+\varepsilon}_{\phi}(t)}=\frac{\Gamma(\varepsilon+1)}{W_{\phi}(\varepsilon+1)}\frac{\phi(\varepsilon)}{\varepsilon}-\sum_{k\geq 1}\frac{\prod_{j=1}^k\lbrb{\phi(-1+\varepsilon+j)-\phi(k)}}{\prod_{j=1}^{k-1}\lbrb{\phi(j)-\phi(k)}}\frac{e^{-\phi(k)t}}{\phi(k)} \frac{\Gamma(\varepsilon+1)}{W_{\phi_{(k)}}(\varepsilon+1)}\frac{\phi_{(k)}(\varepsilon)}{\varepsilon}.
\end{equation*}
Setting $\varepsilon\to 0$ and using $\phi(0)=0$ and $\Wp(1)=1$ for any $\phi$, we conclude the proof in this case.
\end{proof}

\section*{Acknowledgement}
Z. Palmowski acknowledges that the research is partially supported by Polish National Science Centre Grant No. 2018/29/B/ST1/00756.  H. Sariev and M. Savov acknowledge that this study is financed by the European Union - NextGenerationEU, through the National Recovery and Resilience Plan of the Republic of Bulgaria, project No. BG-RRP-2.004-0008.

 \bibliographystyle{plain}
	\bibliography{bibliography.bib}
\end{document}